\documentclass{amsart}

\usepackage{amsmath,amssymb}
\usepackage{amsfonts,amsthm}
\usepackage{sseq}
\usepackage{rotating}
\usepackage{bbm}
\usepackage[all,cmtip]{xy}
\usepackage{amscd}
\usepackage{tikz-cd}
\numberwithin{equation}{section}

\usepackage{subcaption}
\usepackage{mathabx}

\usepackage{tikz}

\usepackage{mathrsfs}
\usepackage{color}
\usepackage{xcolor}

\usepackage{amssymb}
\usepackage[latin1]{inputenc}
\usepackage{graphicx}
\usepackage{dsfont}
\usepackage{color}
\usepackage{hyperref}

\newtheorem{theorem}{Theorem}[section]

\newtheorem{lemma}[theorem]{Lemma}
\newtheorem{proposition}[theorem]{Proposition}
\newtheorem{corollary}[theorem]{Corollary}

\theoremstyle{definition}
\newtheorem{definition}[theorem]{Definition}

\newtheorem{example}[theorem]{Example}

\theoremstyle{remark}
\newtheorem{remark}[theorem]{Remark}

\theoremstyle{theorem}
\newtheorem{TheoremA}{Theorem}

\newtheorem*{task}{Task}
\newtheorem*{question}{Question}

\newcommand*{\longhookrightarrow}{\ensuremath{\lhook\joinrel\relbar\joinrel\rightarrow}}

\newcommand{\ow}{\omega}

\newcommand{\C}{{\mathbb{C}}}

\newcommand{\R}{{\mathbb{R}}}

\newcommand{\Z}{{\mathbb{Z}}}

\renewcommand{\epsilon}{\varepsilon}
\renewcommand{\theta}{\vartheta}

\renewcommand{\S}{S}

\DeclareMathOperator{\Cl}{Clif}
\DeclareMathOperator{\Crit}{Crit}
\DeclareMathOperator{\Stab}{Stab}

\DeclareMathOperator{\FS}{{FS}}
\DeclareMathOperator{\Lie}{Lie}

\DeclareMathOperator{\id}{id}
\DeclareMathOperator{\im}{Im}

\DeclareMathOperator{\Symp}{Symp}
\DeclareMathOperator{\Ann}{Ann}

\DeclareMathOperator{\Fix}{Fix}

\DeclareMathOperator{\GL}{GL}

\DeclareMathOperator{\Aut}{Aut}

\newcommand{\proofend}{\hspace*{\fill} $\Box$\\}
\newcommand{\proofof}[1]{\noindent {\it Proof of #1. }}

\usepackage{graphicx}

\usepackage{etex}

\begin{document}
\title{On the topology of real Lagrangians in toric symplectic manifolds}
\author{Jo\'{e} Brendel, Joontae Kim, and Jiyeon Moon}

\address{Institut de Math\'ematiques, Universit\'e de Neuch\^atel, Rue Emile-Argand 11, 2000 Neuch\^atel, Switzerland}
\email{joe.brendel@unine.ch}

\address{School of Mathematics, Korea Institute for Advanced Study, 85 Hoegiro, Dongdaemun-gu, Seoul 02455, Republic of Korea}
\email{joontae@kias.re.kr}

\address{Department of Mathematics, Ajou University, 206 Worldcup-ro, Suwon 16499, South Korea}
\email{j9746@ajou.ac.kr}
\date{}
\setcounter{tocdepth}{1}
\maketitle

\begin{abstract}
We explore the topology of real Lagrangian submanifolds in a toric symplectic manifold which come from involutive symmetries on its moment polytope. We establish a real analog of the Delzant construction for those real Lagrangians, which says that their diffeomorphism type is determined by combinatorial data. As an application, we realize all possible diffeomorphism types of connected real Lagrangians in toric symplectic del Pezzo surfaces.
\end{abstract}

\tableofcontents

\section{Introduction}\label{sec: intro}
A diffeomorphism $R$ on a symplectic manifold is called an \emph{antisymplectic involution} if it is an involution, $R \circ R = \id$, and if it is antisymplectic, $R^*\omega = -\omega$. Fixed point sets of antisymplectic involutions are either empty or Lagrangian. A Lagrangian $L \subset M$ is called \emph{real} if it is the fixed point set of an antisymplectic involution. We restrict ourselves to the study of real Lagrangians in \emph{toric} symplectic manifolds.

A symplectic manifold $(M,\omega)$ of dimension~$2n$ is called \emph{toric} if it is equipped with an effective Hamiltonian action of the torus $T^n$. Complex projective space $\C P^n$ is a typical example. A classical result by Atiyah--Guillemin--Sternberg \cite{Atiyah, GuilStern} states that the image of the moment map $\mu$ of a Hamiltonian torus action is a convex polytope~$\Delta \subset \Lie(T^n)^* = (\mathfrak{t}^n)^*$, called the \emph{moment polytope}. In the case of $\C P^n$ the moment polytope is the $n$-simplex. Toric manifolds are classified up to equivariant symplectomorphisms by their moment polytope. This was proved by Delzant~\cite{Del}, who starts out with a given polytope satisfying certain properties (called \emph{Delzant polytope}) and gives an explicit description of~$M$ as a symplectic quotient of a symplectic vector space. For details on the Delzant construction, see Section~\ref{sec: delzantconst}.

Let $\mathcal{S}_{\Delta}$ denote the group of lattice-preserving automorphisms of $(\mathfrak{t}^n)^*$ which leave $\Delta$ invariant. We construct antisymplectic involutions from symmetries of the moment polytope.

\begin{TheoremA}
\label{thm: antilift}
Let $(M,\omega)$ be a toric symplectic manifold with moment map $\mu$ and moment polytope $\Delta$. Furthermore, let $\sigma \in \mathcal{S}_{\Delta}$ be an involution of $\Delta$. Then $\sigma$ lifts to an antisymplectic involution $R^{\sigma}$ of $M$,
	\begin{equation} \label{eq: compintro}
	\mu \circ R^{\sigma} = \sigma \circ \mu. 
	\end{equation}
\end{TheoremA}
\noindent
The antisymplectic involution $R^{\sigma}$ we construct is not unique with respect to the property~$(\ref{eq: compintro})$. Henceforth, we will refer to it as the \textit{standard antisymplectic lift of~$\sigma$}. The most basic example for Theorem~\ref{thm: antilift} is the following one. Let $(S^2, \omega)$ be the two-sphere equipped with its area form. The toric structure is given by rotation around a fixed axis and the corresponding moment map is given by projection onto that axis, see Figure~\ref{fig: s2}. Therefore $\Delta$ can be identified with a segment in $\R$. Let $\sigma$ be the only non-trivial involution on $\Delta$ given by the flip around the mid-point of the segment. The corresponding antisymplectic involution on $S^2$ is given by the flip fixing the equator. 

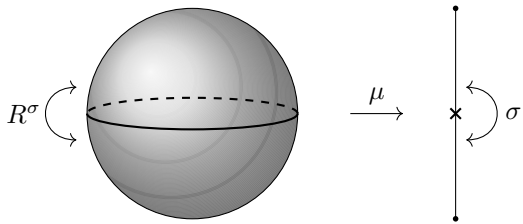
\begin{figure}[h]
\centering
\begin{tikzpicture}[scale = 0.7]
    \shade[ball color = gray!40, opacity = 0.4] (0,0) circle (2cm);
    \draw (5,-2)--(5,2);
    \draw[thick] (4.9,-0.1)--(5.1,0.1);
    \draw[thick] (4.9,0.1)--(5.1,-0.1);
    \fill[fill=black] (5,-2) circle (1.5pt);
    \fill[fill=black] (5,2) circle (1.5pt);
    \draw[black,thick] (0,0) +(180:2) arc (180:360:2 and 0.3);
    \draw[black, thick, dashed] (0,0) +(0:2) arc (0:180:2 and 0.3);
    \draw (0,0) circle (2cm);
    \draw[->] (3,0)--node[above]{$\mu$}(4,0);
    \draw[<->] (-2.2,-0.5)  to [out=180,in=180, looseness=2] (-2.2,0.5);
    \draw[<->] (5.2,-0.5)  to [out=0,in=0, looseness=2] (5.2,0.5);
    \node at (6.1,0) {$\sigma$};
    \node at (-3.2,0) {$R^{\sigma}$};
\end{tikzpicture}
\caption{The flip $\sigma$ on $\Delta$ and the corresponding antisymplectic involution $R^{\sigma}$ on $S^2$.}
\label{fig: s2}
\end{figure}

Even for general toric symplectic manifolds $M$, there is a particularly simple way of understanding the involutions $R^{\sigma}$ in Theorem~\ref{thm: antilift} if we restrict our attention to the open and dense subset $\mu^{-1}(\mathring{\Delta})$ formed by the pre-image of the interior of $\Delta$. In fact, $\mu^{-1}(\mathring{\Delta})$ is equivariantly symplectomorphic to $T^n \times \mathring{\Delta}$, when we equip the latter space with the natural $T^n$-action and the symplectic form coming from the inclusion $T^n \times \mathring{\Delta} \subset T^*T^n \cong T^n \times (\mathfrak{t}^n)^*$. Under this identification, the moment map corresponds to the natural projection $T^n \times \mathring{\Delta} \rightarrow \mathring{\Delta}$. We observe
\begin{enumerate}
\item For any lattice preserving involution $\sigma$ on the moment polytope, the map $(\sigma^T)^{-1} \times \sigma$ defines a symplectomorphism on $T^n \times \mathring{\Delta}$. The transpose $\sigma^T \colon \mathfrak{t}^n \rightarrow \mathfrak{t}^n$ is well-defined on $T^n$, since it preserves the lattice. Hence we obtain a \emph{symplectic} involution on $T^n \times \mathring{\Delta}$.
\item There is a natural antisymplectic involution $R^0$ on $T^n \times \mathring{\Delta}$ given by taking the group inverse on the $T^n$-component. The involution $R^0$ preserves the fibres of $T^n \times \mathring{\Delta} \rightarrow \mathring{\Delta}$.
\end{enumerate}
\noindent
The desired antisymplectic involution $R^{\sigma}$ is obtained by composing the maps obtained in the two observations. Since they commute, the resulting diffeomorphism will indeed be an involution. The main problem with this heuristic argument is extending everything to the singular fibres over $\partial \Delta$. In Section~\ref{sec: lifting} we thus stick to the more conventional approach via Delzant's point of view on toric manifolds. As we shall see, 
this approach also has the merit of providing a method to understand the fixed point set of $R^\sigma$.

In the special case where $\sigma = \id$, we obtain an antisymplectic involution $R^0$ which leaves the moment map invariant $\mu \circ R^0 = \mu$. This involution is widely known in toric geometry, where it corresponds to complex conjugation. Its fixed point set $\Fix R^0$ is the \emph{real locus} of the toric variety, in the case of $\C P^n$ it corresponds to $\R P^n$. Duistermaat \cite{Duist} studied more general real Lagrangians $L=\Fix R$ in Hamiltonian $T^k$-spaces $(M,\ow,\mu)$ for any $k \leqslant n$, which arise as the fixed point set of involutions leaving the moment map invariant, $\mu \circ R = \mu$. He proved that real Lagrangians of this type are tight and have a convex image under the moment map. Tightness of the real Lagrangian $L$ means that for any $\xi\in \mathfrak{t}^n$ the restriction $H_\xi|_L$ of the Hamiltonian function $H_\xi=\langle\mu, \xi \rangle$ is \emph{tight} in the sense that the sum of the Betti numbers of $L$ and the one of the critical set of $H_\xi|_L$ coincide.

Another class of interesting Lagrangians are regular fibres $\mu^{-1}(x) \subset M$ of the toric moment maps. Entov--Poterovich \cite{EP} studied the rigidity of intersections of \emph{Lagrangian fibres}, namely that the barycentric fibre in a closed monotone symplectic manifold cannot be displaced by a Hamiltonian isotopy. See also results of Fukaya--Oh--Ohta--Ono \cite{FOOO}. Theorem~\ref{thm: antilift} shows that the toric fibre $\mu^{-1}(0)$ is real whenever the moment polytope is invariant under the central symmetry $\sigma = -\id$. Indeed, we will see that $\Fix R^{\sigma} \neq \varnothing$, and since $0 \in (\mathfrak{t}^n)^*$ is the only fixed point of $-\id$, the fixed point set of $R^{-\id}$ is the entire fibre $\mu^{-1}(0)$ for dimensional reasons. Under some additional assumptions on $M$, one can show that $\mu^{-1}(0)$ being real is a sufficient condition for $\Delta$ to be invariant under $-\id$. We refer to~\cite{Bre} for details. In a sense, the two classical situations $R=\id$ and $R=-\id$ are opposite to each other and all other $R^{\sigma}$ which we obtain from Theorem~\ref{thm: antilift} are intermediate cases.

The remainder of the paper is dedicated to a topological study of the fixed point sets of the involutions $R^{\sigma}$. The main result in this direction is the so-called \emph{real Delzant construction}, which states that the diffeomorphism type of $L=\Fix(R^\sigma)$ is completely determined by the moment polytope $\Delta$ and the involution $\sigma \in \mathcal{S}_{\Delta}$.

We briefly explain relevant notions in the classical Delzant construction, see Section \ref{sec: delzantconst} for details. The moment polytope $\Delta$ with $k$ facets yields the moment map
$$
\nu \colon \C^k\longrightarrow \mathfrak{k}^*,
$$
where $\mathfrak{k}=\Lie(K)$ is the Lie algebra of the kernel $K$ of the characteristic map $\pi\colon T^k\to T^n$.
By the Marsden--Weinstein theorem, we can reconstruct the toric symplectic manifold,
$$
M \cong \nu^{-1}(0)/K.
$$
In a similar vein, we will define a real analog of the maps $\nu$ and $\pi$, namely
$$
\pi_R\colon \Fix(\rho_{T^k})\longrightarrow \Fix(R_{T^n})
$$
the \emph{real characteristic map}, and
$$
\nu_R\colon \Fix(\rho)\longrightarrow (\mathfrak{k}/\mathfrak{k}_R)^*,
$$
where $\mathfrak{k}_R\subset \mathfrak{k}$ is the Lie algebra of the kernel $K_R=\ker \pi_R$. Here $\rho$ and $\rho_{T^k}$ are involutions on $\C^k$ and $T^k$, respectively, determined by the involution $\sigma\in \mathcal{S}_\Delta$. See Section \ref{sec: realdelconst} for details.

The main result of the paper is the following {\it real Delzant construction}.
\begin{TheoremA}\label{thm: realdelzant}
Let $(M,\ow,\mu)$ be a toric symplectic manifold and let $R^\sigma$ be the standard antisymplectic involution of $M$ given by the lift of an involution $\sigma\in \mathcal{S}_\Delta$. Then
	the real Lagrangian $L=\Fix(R^\sigma)$ is diffeomorphic to $\nu_R^{-1}(0)/K_R$.
\end{TheoremA}



As a partial generalization of Duistermaat's result, we prove convexity and tightness for the real Lagrangians $L=\Fix(R^\sigma)$.
\begin{TheoremA}\label{thm: tightconvex}
Let $(M,\ow,\mu)$ be a toric symplectic manifold with moment polytope~$\Delta$ and let~$R^\sigma$ be the standard antisymplectic involution of $M$ given by the lift of an involution $\sigma\in \mathcal{S}_\Delta$. Then $\mu(L)=\Fix(\sigma)$ is convex, and	for any $\xi\in \mathfrak{t}^n$ we have
	$$
	\dim H_*(L;\Z_2) = \dim H_*(\Crit(H_\xi|_L);\Z_2),
	$$
where $H_\xi$ is the smooth function $\langle \mu,\xi \rangle$ on $M$ and $\Crit(H_\xi|_L)$ denotes the set of critical points of $H_\xi|_L$.
\end{TheoremA}
In particular, both Theorem~\ref{thm: realdelzant} and \ref{thm: tightconvex} imply that $L=\Fix(R^\sigma)$ is not empty.
Example \ref{ex: tightconvexfails} shows that the tightness and the convexity, in general, fail if the real Lagrangian is not of the form $\Fix R^{\sigma}$.

As an application, we show that the class of real Lagrangians that come from involutive symmetries of moment polytopes provides a starting point for the classification of real Lagrangians in toric symplectic del Pezzo surfaces.

Recall that the symplectic del Pezzo surfaces $Q=S^2\times S^2$ and $X_k=\C P^2\# k\overline{\C P^2}$ for~$0\le k \le 3$ are \emph{monotone} and toric. Being monotone means that their first Chern class is positively proportional to the cohomology class of the symplectic form, see Section~\ref{sec: delpezzo}. Using Smith theory, the Arnold lemma, and homological obstructions for Lagrangians, we show that any real Lagrangian $L$ in a toric symplectic del Pezzo surface $M$ must be diffeomorphic to one of cases listed in Table~\ref{tab: delpezzo}. We then realize all of these possible cases as fixed point sets of lifted antisymplectic involutions $R^{\sigma}$. The real Delzant construction will be used to determine their diffeomorphism types.  We refer to Section \ref{sec: delpezzo} for details.

\begin{TheoremA}
Let $L$ be a connected real Lagrangian submanifold of a toric symplectic del Pezzo surface $M$. Then $L$ is diffeomorphic to one of the surfaces in Table~\ref{tab: delpezzo}, and each of these diffeomorphism types is realized as the fixed point set $\Fix R^{\sigma}$ of an antisymplectic involution from Theorem~\ref{thm: antilift}.
\end{TheoremA}

\begin{table}[hbt]
  \begin{tabular}{c|cccccc}
    $M$ & $L=\Fix(R)$ & & &&&\\ 
    \hline\rule{0pt}{3ex}
   $S^2\times S^2$ & $ S^2$ & $ T^2$ & &&&\\\rule{0pt}{3ex}
    $X_0$ & & & $ \R P^2$ &&&\\\rule{0pt}{3ex}
    $X_1$ & &  & &$ \R P^2\# \R P^2$&&\\\rule{0pt}{3ex}
    $X_2$ & & & $ \R P^2$ && $ \#_3\R P^2$&\\\rule{0pt}{3ex}
    $X_3$ & $ S^2$ & $ T^2$ & & $ \R P^2\# \R P^2$ && $ \#_4\R P^2$
  \end{tabular}

  \caption{The diffeomorphism types of connected real Lagrangians in toric symplectic del Pezzo surfaces.}
    \label{tab: delpezzo}
\end{table}
\begin{task}
	Classify connected real Lagrangians in toric symplectic del Pezzo surfaces up to Hamiltonian isotopy.
\end{task}
For fixed diffeomorphism type of the real Lagrangian, uniqueness up to Hamiltonian isotopy is known for $S^2$ in $S^2\times S^2$ and $X_3$ by \cite{HindS2S2} and \cite{Evans}, and for $\R P^2$ in $X_0$ by \cite{LiWu}. Indeed, every real Lagrangian in a monotone symplectic manifold is monotone and there are no exotic monotone submanifolds in these cases. 
\begin{question}
	Is every connected real Lagrangian with fixed diffeomorphism type in toric del Pezzo surfaces unique up to Hamiltonian isotopy?
\end{question}
Since $S^2\times S^2$ and $X_3$ admit infinitely many exotic monotone Lagrangian tori \cite{Vianna}, a positive answer to this question would crucially depend on the submanifold being real.
\section{Basic geometry}
We refer to \cite{Sja}, \cite{Audin} and \cite[Chapter 5]{McduffSalamon} for (real) symplectic geometry and Hamiltonian torus actions. 
\subsection{Basics}

Let $(M,\omega)$ be a symplectic manifold and let the $n$-torus $T^n$ act on $M$ by symplectomorphisms. We denote this action by $(t,p) \mapsto t.p$ for $p\in M$ and $t\in T^n$ and the corresponding Lie algebra by $\mathfrak{t}^n = \Lie(T^n)$. The associated infinitesimal action $\mathfrak{t}^n \rightarrow \Gamma(TM)$ is defined by
	\begin{equation}
	\xi\longmapsto X_{\xi}, \quad \left(X_{\xi}\right)_p := \left.\frac{d}{ds}\right\vert_{s=0} \exp(s\xi).p,\quad p\in M.
	\end{equation}
A symplectic $T^n$-action on a symplectic manifold $M$ is called \emph{Hamiltonian} if there exists a smooth map $\mu : M\to (\mathfrak{t}^n)^*$ such that
\begin{itemize}
	\item[1)] for each $\xi\in \mathfrak{t}^n$ we have $d\langle \mu, \xi\rangle=\iota_{X_\xi}\ow$, where $\langle\cdot,\cdot \rangle$ denotes the natural pairing between $\mathfrak{t}^n$ and $(\mathfrak{t}^n)^*$,
	\item[2)] the map $\mu$ is invariant under the $T^n$-action, i.e. $\mu(t.p)=\mu(p)$ for all $t \in T^n$ and $p \in M$.
\end{itemize}
The map $\mu$ is called a \emph{moment map} of the Hamiltonian $T^n$-action.

\begin{definition}
A triple $(M,\omega,\mu)$ is called a {\bf Hamiltonian $T^n$-space} if $(M,\ow)$ is a symplectic manifold equipped with a Hamiltonian $T^n$-action and $\mu\colon M\to (\mathfrak{t}^n)^*$ is a moment map associated to the action.
\end{definition}

\noindent
The equation $d\langle \mu , \xi \rangle = \iota_{X_{\xi}}\omega$ means that the Hamiltonian flow of $ \langle \mu , \xi \rangle \in C^{\infty}(M) $ at time $t$ corresponds to the action of $\exp(t\xi)$ on $M$. Furthermore, this equation can be used to prove the following geometric properties of the moment map
\begin{eqnarray}
	(\ker d\mu\vert_p)^{\omega} &=& T_p(T^np) ,\\
	\Ann(\im d\mu\vert_p) &=& \Lie(\Stab(p)).			\label{eq:momentstab}
\end{eqnarray}  

\noindent
A classical result by Atiyah--Guillemin--Sternberg states that the image of $\mu$ is a convex polytope in $(\mathfrak{t}^n)^*$, called the \emph{moment polytope}.

\begin{remark}
\label{rk:momentnorm}
{\rm
As for Hamiltonians in general, adding a constant vector to the moment map does not change the group action it generates. We choose the normalization $\int_M \mu \omega^n= 0 \in (\mathfrak{t}^n)^*$ for compact $M$ unless otherwise stated.
}
\end{remark}

\noindent
The \emph{standard lattice} $\mathfrak{t}^n_{\Z}$ is defined as the kernel of the exponential map $\exp:\mathfrak{t}^n \rightarrow T^n$. Furthermore, the group formed by the automorphisms of $\mathfrak{t}^n$ which preserve the standard lattice will be denoted by $\Aut_{\Z}\mathfrak{t}^n$. The dual of the standard lattice is defined by
\begin{equation*}
(\mathfrak{t}^n_{\Z})^* = \{\eta \in \mathfrak{t}^* \,\vert\, \langle \eta , x \rangle \in \Z \text{ for all } x \in \mathfrak{t}^n_{\Z} \}.
\end{equation*}
The corresponding group $\Aut_{\Z}\mathfrak({\mathfrak{t}}^n)^*$ is defined similarly. 

\begin{remark} 
\label{rk:liftingonT}
{\rm
Since $T^n \cong \mathfrak{t}^n / \mathfrak{t}^n_{\Z}$, any element $\alpha \in \Aut_{\Z}\mathfrak{t}^n$ induces a group automorphism $A$ of $T^n$. Conversely, for any group automorphism $A$ of $T^n$, its differential $A_*$ belongs to $\Aut_{\mathbb{Z}}\mathfrak{t}^n$.
}
\end{remark}

\noindent
Recall that for a given Hamiltonian $T^n$-space, one can perform symplectic reduction on certain level sets of the moment map in order to obtain a new symplectic manifold. See~\cite{Can} for details. 

\begin{proposition} \label{prop:red} Let $(M,\omega,\mu)$ be a Hamiltonian $T^n$-space and $0 \in (\mathfrak{t}^n)^*$ a regular value of $\mu$ such that $T^n$ acts freely on the corresponding level set $\mu^{-1}(0)$. Then the quotient 
\begin{equation*}
	\widehat{M} = \mu^{-1}(0)/T^n
\end{equation*}
carries a unique symplectic structure $\widehat{\omega}$ such that 
\begin{equation*}
\iota^* \omega = p^*\widehat{\omega}
\end{equation*}
where $\iota : \mu^{-1}(0) \hookrightarrow M$ is the natural inclusion and $p : \mu^{-1}(0) \twoheadrightarrow \widehat{M}$ is the natural projection. 
\end{proposition}

\noindent
The space $(\widehat{M},\widehat{\omega})$ is called \emph{symplectic quotient} or \emph{Marsden--Weinstein quotient} at the level $0$. This construction is best summarized by the reduction diagram

\begin{center}
\begin{tikzcd}
	\left(\mu^{-1}(0), \iota^*\omega = p^*\widehat{\omega} \right) \arrow[hook]{r}{\iota} \arrow[two heads]{d}{/T^n}[swap]{p}
	&  \left( M, \omega \right)  \\
	(\widehat{M}, \widehat{\omega}).
\end{tikzcd}
\end{center}

\subsection{Compatible maps}

Let $(M,\omega,\mu)$ be a Hamiltonian $T^n$-space. We will define a notion of compatibility between the torus action and a given diffeomorphism $\varphi$ of $M$ which either preserves or reverses the symplectic form, i.e. which is either symplectic or antisymplectic. In order to treat both cases simultaneously, we attach a sign $\epsilon(\varphi) \in \{-1,1\} $ to the diffeomorphism $\varphi$ such that 
\begin{equation*}
	\varphi^* \omega = \epsilon(\varphi) \omega.
\end{equation*}

\begin{proposition} 
\label{prop:compatibility}
Let $\varphi$ be a diffeomorphism of a Hamiltonian $T^n$-space $(M,\omega,\mu)$ satisfying $\varphi^* \omega = \epsilon(\varphi) \omega$ for $\epsilon(\varphi) \in \{-1,1\}$. Then the following are equivalent.
	\begin{itemize}
		\item[1)] There is a group automorphism $\tau: T^n \rightarrow T^n$ such that
			\begin{equation}
				\label{eq:comp1}
				\varphi(t.p) = \tau(t).\varphi(p), \quad p\in M,\, t \in T^n;
			\end{equation}
		\item[2)] There is a map $\sigma \in \Aut_{\Z}(\mathfrak{t}^n)^*$ such that 
			\begin{equation} 
				\label{eq:comp2}
				\mu \circ \varphi = \sigma \circ \mu;
			\end{equation}
		\item[3)] There is a map $\alpha \in \Aut_{\Z}\mathfrak{t}^n$ such that 
			\begin{equation}
				\label{eq:comp3}
				\varphi^{-1}_*(X_{\xi} \circ \varphi) = X_{\alpha(\xi)}, \quad \xi \in \mathfrak{t}^n.
			\end{equation}
	\end{itemize}
Furthermore, if the statements are true, then the above maps are related by
	\begin{equation}
		\label{eq:comprel}
		\tau^{-1}_* = \epsilon(\varphi)\sigma^* = \alpha.
	\end{equation}
\end{proposition}
\proof First suppose that $\epsilon(\varphi)=1$.\\
We will show that both $1)$ and $2)$ are equivalent to the infinitesimal condition $3)$. Since the exponential map of $T^n$ is surjective, $1)$ is equivalent to
\begin{equation*}
	\varphi(\exp s \xi . p) = \tau(\exp s \xi).\varphi(p), \quad p \in M, \, \xi \in \mathfrak{t}^n.
\end{equation*}
Differentiating with respect to $s$ and rearranging terms, we obtain
\begin{equation*}
	\varphi_*^{-1}(X_{\xi})_{\varphi(p)} = (X_{\tau^{-1}_*\xi})_p, \quad p \in M, \, \xi \in \mathfrak{t}^n.
\end{equation*}
The equivalence of $1)$ and $3)$ follows by defining, with the help of Remark \ref{rk:liftingonT}, $\alpha := \tau^{-1}_* \in \Aut_{\Z}\mathfrak{t}^n$ and conversely by defining $\tau$ as the automorphism obtained by lifting $\alpha^{-1}$ to $T^n$.

In order to prove the equivalence of $2)$ and $3)$, recall that given a Hamiltonian $H$, its vector field $X_H$ transforms under a symplectomorphism $\varphi$ to the Hamiltonian vector field $X_{H \circ \varphi} = \varphi^{-1}_*(X_H \circ \varphi)$. Since $X_{\xi}$ is the vector field corresponding to the Hamiltonian function $H=\langle \mu , \xi \rangle$, identity $\eqref{eq:comp3}$ can be rewritten as 
\begin{equation*}
	X_{\langle \mu \circ \varphi, \xi \rangle} = X_{\langle \mu , \alpha(\xi) \rangle},
\end{equation*}
\noindent
which is equivalent to 
\begin{equation*}
	\mu \circ \varphi = \alpha^* \circ \mu.
\end{equation*}

The case $\epsilon(\varphi) = -1$ can be proved similarly. The only notable difference is the fact that if $\varphi$ is antisymplectic, then the Hamiltonian vector fields transform as follows,
\begin{equation*}
	X_{H \circ \varphi} = - \varphi^{-1}_*(X_H \circ \varphi).
\end{equation*}
This accounts precisely for the additional minus sign in equation \eqref{eq:comprel}.
\proofend

\begin{remark}
\label{rk:involutioncomp}
In the antisymplectic case, we will mostly work with involutions, i.e. diffeomorphisms $R:M \rightarrow M$ satisfying $R^*\omega = -\omega$ and $R^2 = \id$. In this case, the maps $\tau$, $\alpha$ and $\sigma$ are involutions as well.
\end{remark}

\begin{definition}
\label{def:compatibility}
An (anti-)symplectic diffeomorphism $\varphi$ on a Hamiltonian $T^n$-space $(M,\omega,\mu)$ is called {\bf compatible} if one of the equivalent conditions in Proposition~\ref{prop:compatibility} holds.
\end{definition}

\begin{remark}
This compatibility condition is a special case of the notion of \emph{real Hamiltonian $G$-manifold} given in \cite{Sja}, which contains many examples. These ideas go back to Duistermaat's work \cite{Duist}, who considered the case where $\tau(t)=t^{-1}$.
\end{remark}

\noindent
In case the Hamiltonian $T^n$-space admits symplectic reduction, a given compatible (anti-)symplectic map yields an (anti-)symplectic map on the symplectic quotient.

\begin{proposition} 
\label{prop:symplecticred}
Let $\varphi$ be a compatible diffeomorphism on a Hamiltonian $T^n$-space $(M,\omega,\mu)$ satisfying $\varphi^* \omega = \epsilon(\varphi) \omega$ for $\epsilon(\varphi) \in \{-1,1\}$. Furthermore, suppose that $M$ admits symplectic reduction at the level $0 \in (\mathfrak{t}^n)^*$. Then $\varphi$ induces a diffeomorphism $\widehat{\varphi} : \widehat{M} \rightarrow \widehat{M}$ on the symplectic quotient satisfying $\widehat{\varphi}^* \widehat{\omega} = \epsilon(\varphi) \widehat{\omega}$ such that the following diagram commutes,

\begin{center}
\begin{tikzcd}
	\mu^{-1}(0)  \arrow[loop above]{}{\varphi\vert_{\mu^{-1}(0)}} \arrow[hook]{r}{\iota} \arrow[two heads]{d}{/T^n}[swap]{p}
	&   M  \arrow[loop above]{}{\varphi} \\
	\widehat{M} \arrow[loop left]{}{\widehat{\varphi}}
\end{tikzcd}
\end{center}

\end{proposition}

\proof 

The diffeomorphism $\varphi$ preserves the level set $\mu^{-1}(0)$, as can be read off from \eqref{eq:comp2}. Furthermore, the restriction $\varphi\vert_{\mu^{-1}(0)}$ descends to $\widehat{M}$ by equation \eqref{eq:comp1} to yield a diffeomorphism $\widehat{\varphi}$. Since $\widehat{\omega}$ is defined by $\iota^*\omega = p^* \widehat{\omega}$, we can compute
\begin{eqnarray*}
p^* \widehat{\varphi}^* \widehat{\omega} 
&=& \varphi\vert_{\mu^{-1}(0)}^* p^* \widehat{\omega} \\
&=& \varphi\vert_{\mu^{-1}(0)}^* \iota^* \omega \\
&=& \iota^* \varphi^* \omega \\
&=& \iota^* \left( \epsilon(\varphi) \omega \right) \\
&=& p^* \left( \epsilon(\varphi) \widehat{\omega} \right). 
\end{eqnarray*}
Since $p$ is a surjective submersion, this implies that $\widehat{\varphi}^* \widehat{\omega} = \epsilon(\varphi)\widehat{\omega}$. 
\proofend

\section{Toric symplectic manifolds}

We refer to \cite{Audin}, \cite{Can}, \cite{Del}, \cite{Gui2} or \cite{Mcduffmono} for details on toric symplectic manifolds and the Delzant construction. 

\subsection{Basics}

Toric symplectic manifolds are a special case of Hamiltonian $T^n$-spaces. 

\begin{definition}
A Hamiltonian $T^n$-space $(M,\omega,\mu)$ is called {\bf toric} if $n = \frac{1}{2} \dim M $ and the action is effective.
\end{definition}

\noindent
In the case of toric symplectic manifolds the moment map is a quotient map for the torus action, and our choice of normalization in Remark~\ref{rk:momentnorm} implies that the barycentre of its moment polytope $\Delta$ lies at $0 \in (\mathfrak{t}^n)^*$. Furthermore, by a classical result of Delzant, $\Delta = \mu(M)$ takes a particular form and is, in fact, a sufficient datum to reconstruct $(M,\omega,\mu)$ along with its $T^n$-action up to equivariant symplectomorphisms. We will recall Delzant's result and some of the facts surrounding it, since these will be used later on.\\

\noindent 
Let $\Delta \subset (\mathfrak{t}^n)^* $ be a rational polytope with respect to the standard lattice $(\mathfrak{t}^n_{\Z})^*$ bounded by $k$ hyperplanes. A lattice vector $v \in \mathfrak{t}_{\Z}^n$ is called \emph{primitive} if it cannot be written as a non-trivial integer multiple of another lattice vector. Equivalently, a primitive vector is the first intersection of the line it spans with the standard lattice. We can describe $\Delta$ in terms of primitive vectors $v_i \in \mathfrak{t}_{\Z}^n$ and a set of numbers $\kappa_i \in \R$,
\begin{equation}
	\label{eq:polytopedescription}
	\Delta = \{ \eta \in (\mathfrak{t}^n)^* \, \vert \, \langle \eta , v_i \rangle \leq \kappa_i \}. 
\end{equation} 

\noindent 
After identifying $\mathfrak{t}^n$ and $(\mathfrak{t}^n)^*$ with $\R^n$ by the choice of a basis, the vectors $v_i$ correspond to outward pointing primitive normal vectors to the facets. The constants $\kappa_i$ measure the affine distance of the facets to the origin. Details can be found in~\cite{Mcduffmono}.

\subsection{The Delzant construction}\label{sec: delzantconst}

\begin{definition} A rational polytope $\Delta \subset (\mathfrak{t}^n)^*$ is called {\bf Delzant} if each of its vertices is formed by the intersection of $n$ hyperplanes whose primitive normal vectors form a $\Z$-basis of $\mathfrak{t}_{\Z}^n$.
\end{definition}

\begin{remark}
\label{rk:alternativedelzant}
If we identify $\mathfrak{t}^n_{\Z}$ with $\Z^n$, the Delzant condition on polytopes is equivalent to requiring that the set of primitive normal vectors at any given vertex can be mapped to the standard basis $\{e_1,...,e_n\} \subset \Z^n$  by an element of $\GL(n,\Z)$.
\end{remark}

\begin{theorem}[\cite{Del}] \label{thm:delzant}
The moment polytope of a toric symplectic manifold is Delzant and there is a bijective correspondence between Delzant polytopes up to $\Aut_{\Z} (\mathfrak{t}^n)^* $-action and toric symplectic manifolds up to $T^n$-equivariant symplectomorphisms. 
\end{theorem}

\noindent
Furthermore, Delzant gave an explicit construction of the toric symplectic manifold $(M,\omega,\mu)$, starting from a given Delzant polytope $\Delta  = \{ \eta \in (\mathfrak{t}^n)^* \, \vert \, \langle \eta , v_i \rangle \leq \kappa_i \} $ such that $\mu(M)=\Delta$. The desired manifold $M$ is obtained as a symplectic quotient of $(\C^k,\omega_0)$. Since we will heavily rely on the details of this construction, it will be recalled here. Details can be found in the original paper~\cite{Del}, or in~\cite{Can} and~\cite{Gui2}.\\

Let $\Delta \subset (\mathfrak{t}^n)^*$ be a Delzant polytope. Up to a translation, we can assume that the normalization convention from Remark~\ref{rk:momentnorm} holds. Via the description $(\ref{eq:polytopedescription})$, the polytope $\Delta$ uniquely defines a set of pairs $\{(v_i,\kappa_i)\}_{i\in \{1,...,k\}}$. The {\bf characteristic map} associated to $\Delta$ is defined as 
\begin{equation}
	\label{eq:charmap}
	\pi : \mathfrak{t}^k \rightarrow \mathfrak{t}^n, \quad \pi(e_i) = v_i,
\end{equation}
where $e_i$ denotes the $i$-th standard basis vector of $\mathfrak{t}^k \cong \R^k$. The characteristic map is thus a linear map of full rank $n$. Furthermore, it maps $\mathfrak{t}_{\Z}^k$ to $\mathfrak{t}_{\Z}^n$, since the vectors $v_i$ are integral. Hence it descends to the respective tori to yield a map $T^k \rightarrow T^n$, which we again denote by $\pi$. Let $K = \ker \pi \subset T^k$ and denote by $\mathfrak{k}$ and $\mathfrak{k}^*$ its Lie algebra and its dual Lie algebra. We get three short exact sequences,
\begin{alignat}{7} 
\label{eq:ses}
1 	&\rightarrow 
&& \;\; K 
&& \stackrel{j}{\hookrightarrow}  
&& \;\; T^k 
&& \stackrel{\pi}{\rightarrow} 
&& \; T^n 
&& \rightarrow 1,  \nonumber  \\
0 	&\rightarrow 
&& \;\; \mathfrak{k} 
&& \stackrel{j_*}{\rightarrow} 
&& \;\; \mathfrak{t}^k 
&& \stackrel{\pi}{\rightarrow} 
&& \;\; \mathfrak{t}^n 
&& \rightarrow 0,  \nonumber  \\
0 	&\rightarrow 
&& (\mathfrak{t}^n)^* 
&& \stackrel{\pi^*}{\rightarrow} 
&& (\mathfrak{t}^k)^* 
&& \stackrel{j^*}{\rightarrow} 
&& \;\; \mathfrak{k}^* 
&& \rightarrow 0.
\end{alignat}

\noindent
The desired toric manifold $(M,\omega)$ arises as a symplectic quotient of $(\C^k,\omega_0)$ as follows. The moment map

\begin{equation*}
\nu_0 : \C^k \rightarrow (\mathfrak{t}^k)^*\cong \R^k, \quad (z_1,\ldots,z_k) \mapsto \frac{1}{2}\left(\vert z_1 \vert^2,\ldots, \vert z_k \vert^2 \right) - (\kappa_1,\ldots,\kappa_k)
\end{equation*}

\noindent 
generates the standard $T^k$-action on $\C^k$. The inclusion $j : K \hookrightarrow T^k$ induces a $K$-action on $\C^k$. The moment map corresponding to this $K$-action is given by

\begin{equation}\label{eq: numap}
\nu : \C^k \rightarrow \mathfrak{k}^*, \quad \nu = j^* \circ \nu_0.
\end{equation}

\noindent
One can show that $0 \in \mathfrak{k}^*$ is a regular value of $\nu$ and that $K$ acts freely on $\nu^{-1}(0)$. Thus the conditions for symplectic reduction are satisfied. One can show that the symplectic quotient $\nu^{-1}(0)/K$ with its induced symplectic form is the desired toric manifold $(M,\omega)$. We will briefly describe how the moment map $\mu : M \rightarrow (\mathfrak{t}^n)^*$ defining the toric structure on $M$ is obtained. Combine the symplectic reduction diagram defining $M$ 

\begin{equation} \label{eq:delzantred}
	\begin{tikzcd}
		M 
		& \nu^{-1}(0) \arrow[two heads]{l}[swap]{p}{/K} \arrow[hook]{r}{\iota}
		& \C^k
	\end{tikzcd}
\end{equation}
with the short exact sequence from \eqref{eq:ses} to obtain the commutative diagram 

\begin{equation} \label{eq:toricdiag}
	\begin{tikzcd}
		M 							\arrow[dashed]{dr}[swap]{\mu}
		& \nu^{-1}(0) 				\arrow[two heads]{l}[swap]{p}{/K} \arrow[hook]{r}{\iota}		\arrow[dashed]{d}{\overline{\mu}}
		& \C^k  					\arrow{d}{\nu_0} 												\arrow{dr}{\nu}  \\
		0 							\arrow{r}
		& (\mathfrak{t}^n)^* 			\arrow{r}{\pi^*}
		& (\mathfrak{t}^k)^* 		\arrow{r}{j^*}
		& \mathfrak{k}^*			\arrow{r}
		& 0.
	\end{tikzcd}
\end{equation}

\noindent
The map~$\mu$ is the desired moment map. We will show that both $\mu$ and $\overline{\mu}$ are well-defined maps. Since~$\nu$ is defined as~$j^* \circ \nu_0$, the composition $\nu_0 \circ \iota$ maps $\nu^{-1}(0)$ to the kernel of~$j^*$ and thus, by exactness of the lower row, to the image of~$\pi^*$. Since $\pi^*$ is injective, we obtain a unique map $\overline{\mu}: \nu^{-1}(0) \rightarrow (\mathfrak{t}^n)^*$ with

\begin{equation*}
	\pi^* \circ \overline{\mu} = \nu_0 \circ \iota.
\end{equation*}

\noindent
Since $\iota$ is $K$-equivariant and $\nu_0$ is $T^k$-invariant and therefore in particular $K$-invariant, we obtain that $\overline{\mu}$ is $K$-invariant. Since $\nu^{-1}(0)$ is a $K$-principal bundle with base $M$, this implies that $\overline{\mu}$ factors through $M$ to yield the desired moment map $\mu$ defined by the equation

\begin{equation*}
	\pi^* \circ \mu \circ p = \nu_0 \circ \iota.
\end{equation*}

\begin{example}
\label{ex:s2delzant}
{\rm
Let $n=1$ and take the Delzant polytope $ [-1,1] \subset \R \cong (\mathfrak{t}^1)^*$. Then $k=2$, the outward pointing normal vectors are given by $v_1= (1), v_2=(-1)$, and the corresponding constants are $\kappa_1 = \kappa_2 = 1$. The characteristic map is $\pi = (1,-1)$ and furthermore
\begin{equation*}
	\nu_0 : \C^2 \rightarrow (\mathfrak{t}^2)^*\cong \R^2, \quad (z_1,z_2) \mapsto \left(\frac{1}{2}\vert z_1 \vert^2 - 1 , \frac{1}{2} \vert z_2 \vert^2 -1 \right).
\end{equation*}
Since $K = \ker \pi = \langle (1,1) \rangle$, the map $j^*$ is given by projection to the vector $(1,1)$ and hence
\begin{equation*}
	\nu(z_1,z_2) = \frac{1}{2}\left( \vert z_1 \vert^2 + \vert z_2 \vert^2 \right) - 2.
\end{equation*}
Therefore, the level set $\nu^{-1}(0)$ is a $3$-sphere on which $K \cong \S^1$ acts diagonally. Hence we obtain the Hopf fibration and the quotient is $M \cong \C P^1 \cong \S^2$ with the $K$-equivalence classes $[(z_1,z_2)]_K$ corresponding to the homogeneous coordinates $[z_1 : z_2]$ on $\C P^1$.
}
\end{example}

\section{Lifting symmetries of the moment polytope}
\label{sec: lifting}

Throughout this section, let $(M,\omega)$ be a toric symplectic manifold with moment map $\mu$ and moment polytope $\Delta = \mu(M) \subset (\mathfrak{t}^n)^*$ with normalization $\int_M \mu \omega^n = 0 \in (\mathfrak{t}^n)^*$.

\subsection{Symmetries of the moment polytope \nopunct}
\label{ssec:symm}

\begin{definition}
\label{def: symmetry}
Let $(M,\omega,\mu)$ be a toric symplectic manifold with moment polytope $\Delta$. The group 
	\begin{equation}
		\mathcal{S}_{\Delta}=\{ \sigma \in \Aut_{\Z} (\mathfrak{t}^n)^* \mid \sigma(\Delta)=\Delta  \}
	\end{equation}
is called the {\bf symmetries of $\Delta$}.
\end{definition}

\begin{example}
\label{ex:symmetry}
As discussed in the introduction, there are five toric symplectic del Pezzo surfaces, namely $S^2 \times S^2$ and the blow-ups $X_0,X_1,X_2,X_3$ of $\C P^2$. Their moment polytopes and the corresponding groups $\mathcal{S}_{\Delta}$ are given in Figure~\ref{fig:delpezzo}, where $D_n$ denotes the dihedral group of order $2n$. These groups are readily found by noting that elements of $\GL(2,\Z)$ preserve the affine length of edges. For example after identfying $\Aut_{\Z} (\mathfrak{t}^2)^*$ with $\GL(2,\Z)$, the subgroup $\mathcal{S}_{\Delta_{X_0}} \cong D_3$ is generated by the matrices 
	\begin{equation*}
		\begin{pmatrix}
		-1 & -1 \\
		1 & 0
		\end{pmatrix} 
		\text{ and }
		\begin{pmatrix}
		0 & 1 \\
		1 & 0
		\end{pmatrix}.
	\end{equation*}
\begin{figure}[h]
\begin{subfigure}{0.3\textwidth}
   \centering
\begin{tikzpicture}[scale=0.6]
\draw [thick,fill=blue!10] (-1,-1)--(-1,1)--(1,1)--(1,-1)--(-1,-1);
\draw[step=1.0,black!15,thin] (-2.5,-2.5) grid (2.5,2.5);
\draw [thick] (-1,-1)--(-1,1)--(1,1)--(1,-1)--(-1,-1);
\draw [->] (-2,0)--(2,0);
\draw [->] (0,-2)--(0,2);
\node at (0,1.4)[right]{$S^2 \times S^2$};
\end{tikzpicture}
  \caption{$\mathcal{S}_{\Delta_{S^2\times S^2}} \cong D_4$}
  \label{fig: center}
\end{subfigure}
\begin{subfigure}{0.3\textwidth}
   \centering
\begin{tikzpicture}[scale=0.6]
\draw [thick,fill=blue!10] (-1,-1)--(-1,2)--(2,-1)--(-1,-1);
\draw[step=1.0,black!15,thin] (-2.5,-2.5) grid (2.5,2.5);
\draw [thick] (-1,-1)--(-1,2)--(2,-1)--(-1,-1);
\draw [->] (-2,0)--(2,0);
\draw [->] (0,-2)--(0,2);
\node at (0,1.3)[right]{$X_0$};
\end{tikzpicture}
  \caption{$\mathcal{S}_{\Delta_{X_0}} \cong D_3$}
  \label{fig: center}
\end{subfigure}
\begin{subfigure}{0.3\textwidth}
   \centering
\begin{tikzpicture}[scale=0.6]
\draw [thick,fill=blue!10] (-1,0)--(-1,2)--(2,-1)--(0,-1)--(-1,0);
\draw[step=1.0,black!15,thin] (-2.5,-2.5) grid (2.5,2.5);
\draw [thick] (-1,0)--(-1,2)--(2,-1)--(0,-1)--(-1,0);
\draw [->] (-2,0)--(2,0);
\draw [->] (0,-2)--(0,2);
\node at (0,1.3)[right]{$X_1$};
\end{tikzpicture}
  \caption{$\mathcal{S}_{\Delta_{X_1}} \cong \Z_2$}
  \label{fig: center}
\end{subfigure}

\begin{subfigure}{0.3\textwidth}
   \centering
\begin{tikzpicture}[scale=0.6]
\draw [thick,fill=blue!10] (-1,-1)--(-1,1)--(0,1)--(1,0)--(1,-1)--(-1,-1);
\draw[step=1.0,black!15,thin] (-2.5,-2.5) grid (2.5,2.5);
\draw [thick] (-1,-1)--(-1,1)--(0,1)--(1,0)--(1,-1)--(-1,-1);
\draw [->] (-2,0)--(2,0);
\draw [->] (0,-2)--(0,2);
\node at (0,1.3)[right]{$X_2$};
\end{tikzpicture}
  \caption{$\mathcal{S}_{\Delta_{X_2}} \cong \Z_2$}
  \label{fig: center}
\end{subfigure}
\begin{subfigure}{0.3\textwidth}
   \centering
\begin{tikzpicture}[scale=0.6]
\draw [thick,fill=blue!10] (-1,0)--(-1,1)--(0,1)--(1,0)--(1,-1)--(0,-1)--(-1,0);
\draw[step=1.0,black!15,thin] (-2.5,-2.5) grid (2.5,2.5);
\draw [thick] (-1,0)--(-1,1)--(0,1)--(1,0)--(1,-1)--(0,-1)--(-1,0);
\draw [->] (-2,0)--(2,0);
\draw [->] (0,-2)--(0,2);
\node at (0,1.3)[right]{$X_3$};
\end{tikzpicture}
  \caption{$\mathcal{S}_{\Delta_{X_3}} \cong D_6$}
  \label{fig: center}
\end{subfigure}

\caption{Moment polytope $\Delta$ and $\mathcal{S}_{\Delta}$ for toric symplectic del Pezzo surfaces.}
 \label{fig:delpezzo}
\end{figure}
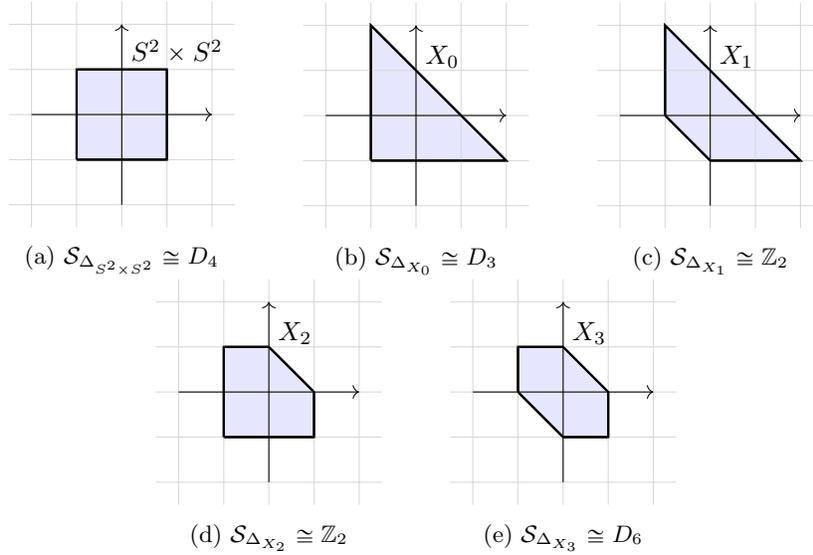
\end{example}
Let $\sigma \in \mathcal{S}_{\Delta}$. Recall from~\eqref{eq:polytopedescription} that we can associate a unique pair $(v_i,\kappa_i)$ to each facet of $\Delta$ such that 
\begin{equation*}
	\Delta = \{ \eta \in (\mathfrak{t}^n)^* \, \vert \, \langle \eta , v_i \rangle \leq \kappa_i \}.
\end{equation*}
\noindent
Applying any $\sigma \in \Aut (\mathfrak{t}^n)^*$ to $\Delta$ yields
	\begin{equation*}
		\sigma(\Delta) = \{ \eta \in (\mathfrak{t}^n)^* \, \vert \,  \langle \eta , (\sigma^{-1})^* v_i \rangle \leq \kappa_i \}.
	\end{equation*}
\noindent
Hence, applying $\sigma$ to the moment polytope amounts to applying $(\sigma^{-1})^* \in \Aut_{\Z}\mathfrak{t}^n$ to its associated normal vectors $v_i$. The hypothesis 
	\begin{equation*}
		\sigma(\Delta) = \Delta
	\end{equation*}
along with the uniqueness of the set of pairs $\{(v_i,\kappa_i)\}_{1\leq i \leq k}$ thus implies that $(\sigma^{-1})^*$ permutes normal vectors. In conclusion, there is a permutation on $k$ elements $\tau \in S_k$ such that
	\begin{eqnarray}
		\label{eq:tau}
		(\sigma^{-1})^*v_i 	&=& v_{\tau(i)},   \\
		\kappa_{i} 			&=& \kappa_{\tau(i)}.
	\end{eqnarray}

\subsection{Lifted symplectomorphisms} 
\label{ssec:liftsymp}

Let $\varphi$ be a compatible symplectomorphism of $M$. In order to clarify notation, the corresponding homomorphism on $T^n$ will be denoted by $\varphi_{T^n}$ from now on. It follows from Proposition~\ref{prop:compatibility} that $\varphi$ descends to a map on $\Delta$. In the following, we will be concerned with proving the opposite direction, namely that symmetries of $\Delta$ can be lifted to symplectomorphisms of $M$.

\begin{lemma}
\label{lem:liftsymp}
Let $\sigma \in \mathcal{S}_{\Delta}$ be a symmetry of the moment polytope of $M$. Then $\sigma$ lifts to a compatible symplectomorphism $\varphi^{\sigma} \in \Symp(M,\omega)$ with 
	\begin{equation}
	\label{eq:phisigmacomp}
		\mu \circ \varphi^{\sigma} = \sigma \circ \mu.
	\end{equation}
\end{lemma}

\noindent
Our construction below will shows that $\varphi^{\sigma \tau} = \varphi^{\sigma} \circ \varphi^{\tau}$. Consequently, the symmetries of $\Delta$ yield a subgroup of the symplectomorphisms of $M$, that we identify with $\mathcal{S}_\Delta$.

\begin{corollary}
The group of symmetries $\mathcal{S}_\Delta$ is a subgroup of $\Symp(M,\ow)$.
\end{corollary}

\begin{remark} 
The lift~$\varphi^{\sigma}$ of~$\sigma$ is not uniquely determined by~\eqref{eq:phisigmacomp}. Nonetheless, any other $\varphi \in \Symp(M,\omega)$ satisfying $\mu \circ \varphi = \sigma \circ \mu$ is closely related to $\varphi^{\sigma}$. This will be discussed in Subsection~\ref{ssec:classification}.
\end{remark}

\proofof{Lemma \ref{lem:liftsymp}} The main idea of the proof is to view $M$ as a symplectic quotient of $\C^k$ via Delzant's construction, then to let the permutation $\tau \in S_k$ defined by~$(\ref{eq:tau})$ act on~$\C^k$ by permutation of coordinates, and to check that this map descends to a symplectomorphism on $M$. Recall that the Delzant construction describes the toric manifold $M$ as a symplectic quotient of $(\C^k,\omega_0)$ by $K = \ker \pi < T^k$. Let $\tau \in S_k$ be the permutation associated to $\sigma$ via equation \eqref{eq:tau}. Since we have $\pi(e_i) = v_i$, applying $(\sigma^{-1})^*$ to $\mathfrak{t}^n$ corresponds to permuting coordinates according to $\tau$ on $\mathfrak{t}^k$,
\begin{equation}
	\label{eq:groupcomp2}
	(\sigma^{-1})^*(\pi(e_i)) = \pi(e_{\tau(i)}).
\end{equation}
\noindent
Notice that this equation holds on the corresponding tori as well, since all maps involved preserve the corresponding lattices. This leads us to define the following permutations of coordinates
\begin{eqnarray}\label{eq: rhomap}
	\Phi : \C^k & \rightarrow & \C^k, \quad (z_1,\ldots,z_k) \mapsto (z_{\tau(1)},\ldots,z_{\tau(k)}), \\\label{eq: rhoTkmap}
	\Phi_{T^k} : T^k & \rightarrow & T^k, \quad (t_1,\ldots,t_k) \mapsto (t_{\tau(1)},\ldots,t_{\tau(k)}).
\end{eqnarray}
The map $\Phi$ is symplectic and compatible with the $T^k$-action,
\begin{equation} 
	\label{eq:rhocomp}
	\Phi(t.z) = \Phi_{T^k}(t).\Phi(z), \quad t \in T^k, \; z \in \C^k.
\end{equation}

\noindent
We will prove that $\Phi$ is $K$-compatible as well. This allows us to apply Proposition~\ref{prop:symplecticred} to the reduction in the Delzant construction, which yields an induced symplectomorphism $\varphi^{\sigma} \in \Symp(M,\omega)$ on the quotient $M$ such that the following diagram commutes
\begin{equation} 
	\label{eq:spacecomp}
	\begin{tikzcd}
	M \arrow[loop above]{d}{\varphi^{\sigma}}
	& \nu^{-1}(0) \arrow[hook]{r} \arrow[loop above]{}{\Phi \vert_{\nu^{-1}(0)}} \arrow[two heads]{l}{/K}
	&  (\C^k, \omega_0)  \arrow[loop above]{}{\Phi}.
	\end{tikzcd}
\end{equation}

\noindent
Since $\Phi$ is compatible with the full $T^k$-action, it suffices to prove that $\Phi_{T^k}$ preserves the subgroup $K$ in order to prove that it is $K$-compatible. The identity \ref{eq:groupcomp2} now reads
\begin{equation}
	\label{eq:groupcomp}
	\sigma^* \circ \pi = \pi \circ \Phi^{-1}_{T^k}.
\end{equation}

\noindent
Recall that by definition $K = \ker \pi$ and hence
\begin{equation}
\label{eq:kpreserved}
\Phi_{T^k} (K) = \Phi_{T^k} (\ker \pi) = \ker (\pi \circ \Phi_{T^k}^{-1}) = \ker (\sigma^* \circ \pi) = \ker \pi = K,
\end{equation}
where we have used elementary properties of $\ker(\cdot)$. This proves that $\Phi$ is $K$-compatible and therefore that $\varphi^{\sigma} : M \rightarrow M$ is a well-defined symplectomorphism. \\

We now show that the symplectomorphism $\varphi^{\sigma}$ is compatible with the Hamiltonian $T^n$-action on $M$ and induces the initially chosen symmetry $\sigma \in \mathcal{S}_{\Delta}$ on $\Delta$,
\begin{equation*} 
	\label{eq:momentcomp}
	\mu \circ \varphi^{\sigma} = \sigma \circ \mu.
\end{equation*}

\noindent
Since $\Phi$ is $T^k$-compatible, Proposition \ref{prop:compatibility} yields
\begin{equation}
	\nu_0 \circ \Phi = (\Phi_{T^k}^{-1})^* \circ \nu_0.
\end{equation}
Adding all of the above maps to~$(\ref{eq:toricdiag})$ we obtain the following commutative diagram.
\begin{equation*} 
	\begin{tikzcd}
		M 						\arrow[loop above]{d}{\varphi^{\sigma}}				\arrow{dr}{\mu}
		& \nu^{-1}(0) 			\arrow[loop above]{}{\Phi\vert_{\nu^{-1}(0)}}			\arrow[two heads]{l}[swap]{p}{/K} \arrow[hook]{r}{\iota}		
		& \C^k  				\arrow[loop above]{}{\Phi}								\arrow{d}{\nu_0}	\arrow{dr}{\nu}  \\
		0 													\arrow{r}
		& (\mathfrak{t}^n)^* 	\arrow[loop below]{d}{\sigma}							\arrow{r}{\pi^*}
		& (\mathfrak{t}^k)^* 	\arrow[loop below]{}{(\Phi_{T^k}^{-1})^*}				\arrow{r}{j^*}
		& \mathfrak{k}^*		\arrow{r}
		& 0
	\end{tikzcd}
\end{equation*}
\noindent
Since the moment map $\mu$ was defined by this diagram, the claim follows.
\proofend

\begin{example} 
{\rm 
Let $M = S^2 = \C P^1$ be equipped with its standard toric structure given in Example \ref{ex:s2delzant}. The moment polytope is $\Delta = [-1,1]$ and there is only one non-trivial symmetry $\sigma \in \mathcal{S}_{\Delta}$ given by $z\mapsto -z$. Since $\sigma$ exchanges the two normal vectors $v_1 = (1)$ and $v_2 = (-1)$, the corresponding permutation is the non-trivial permuation on two elements and $\Phi(z_1,z_2) = (z_2,z_1)$. This yields the lifted symplectomorphism $\varphi^{\sigma}([z_1 : z_2]) = [z_2 : z_1]$ on $S^2 \cong \C P^1$. When we view $S^2$ as embedded in $\R^3$, this symplectomorphism corresponds to the map $(x,y,z) \mapsto (x,-y,-z)$, which obviously induces $\sigma$ on $\Delta$.
}
\end{example}

\begin{example}
{\rm
Let $M= \C P^2$ be equipped with its standard toric structure. Then $\mathcal{S}_{\Delta}$ is isomorphic to the dihedral group $D_3$. The generators given in Example~\ref{ex:symmetry} correspond to the symplectomorphisms
	\begin{equation*}
		[z_0 : z_1 : z_2] \mapsto [z_2 : z_0 : z_1] \; \text{ and } \; [z_0 : z_1 : z_2] \mapsto [z_0 : z_2 : z_1].
	\end{equation*}
}
\end{example}

\subsection{Lifted antisymplectic involutions}
\label{ssec:liftantisymp}
We restrict our attention to the case where $\sigma \in \mathcal{S}_{\Delta}$ is an involution and prove Theorem~\ref{thm: antilift} stated in the introduction.

\begin{theorem}
\label{thm:liftanti}
Let $\sigma \in \mathcal{S}_{\Delta}$ be an involution on the Delzant polytope of a toric symplectic manifold $(M,\omega,\mu)$. Then there is an antisymplectic involution $R^{\sigma}$ such that 
	\begin{equation}
	\label{eq:comprsigma}
		\mu \circ R^{\sigma} = \sigma \circ \mu.
	\end{equation}
\end{theorem}
\noindent
The idea of proof is as follows. By Lemma~\ref{lem:liftsymp}, the involution $\sigma$ lifts to a symplectic involution $\varphi^{\sigma}$ on $M$. To make this map antisymplectic, we will compose it with a standard antisymplectic involution $R^0$, coming from the toric structure on $M$, and apply the following remark

\begin{remark}
\label{rk:commuteinvol}
Let $S$ be a symplectic involution and $R$ an antisymplectic involution on a symplectic manifold $M$, such that $S$ and $R$ commute. Then $S \circ R$ defines an antisymplectic involution on $M$.
\end{remark}
\noindent
In order to prove the theorem, we will first construct the antisymplectic involution $R^0$ on $M$. This construction is well-known, see for example \cite[Definition 2.6]{Gui} or \cite[Section 2.6]{Haug}. We will prove it for the convenience of the reader and in order to expose its relation to the Delzant construction. 

\begin{proposition}
\label{prop:stdantiinvol}
Let $(M,\omega,\mu)$ be a toric symplectic manifold. Then there is an antisymplectic involution $R^0$ which leaves the moment map invariant, i.e. 
	\begin{equation} 
	\label{eq:rzerocomp}
		\mu \circ R^0 = \mu.
	\end{equation}
\end{proposition}
\proof Let 
\begin{equation*} 
	\begin{tikzcd}
		M 
		& \nu^{-1}(0) \arrow[two heads]{l}[swap]{\pi}{/K} \arrow[hook]{r}{\iota}
		& \C^k
	\end{tikzcd}
\end{equation*}
denote the reduction diagram of the Delzant construction. Take the standard antisymplectic involution on $\C^k$ defined by complex conjugation $\rho^0(z_1,...,z_k) = (\overline{z}_1,...,\overline{z}_k)$. It descends to an antisymplectic involution on the quotient $M$ satisfying equation~\ref{eq:rzerocomp}. Indeed, we have
\begin{equation*}
	\rho^0(t.z) = t^{-1}.\rho^0(z), \quad t \in T^k, z \in \C^k. 
\end{equation*}
\noindent
Hence, $\rho^0$ is compatible in the sense of Definition~\ref{def:compatibility} and the claim follows from Proposition~\ref{prop:symplecticred}, since~$\rho^0$ leaves the moment map~$\nu_0$ invariant. 
\proofend

\begin{remark}
\label{rk:reallocus}
In the context of algebraic geometry, the fixed point set $\Fix R^0$ of the above involution is commonly referred to as \emph{real locus} of the toric variety $M$. See also \cite{DJ} for a topological generalization of real toric varieties.
\end{remark}

\begin{remark}
Alternatively, $R^0$ can be viewed as follows. The pre-image $\mu^{-1}(\mathring{\Delta}) \subset M$ of the interior of $\Delta$ is equivariantly symplectomorphic to the product $T^n \times \mathring{\Delta} \subset T^*T^n$ equipped with the natural symplectic form. In the language of Hamiltonian dynamics, this symplectomorphism is referred to as \emph{global action-angle coordinates}, since it corresponds to Arnold--Liouville coordinates on an open dense subset of $M$. There is a natural antisymplectic involution on $T^n \times \mathring{\Delta}$ given by group inversion on the $T^n$-component. This involution extends to all of $M$ and corresponds to $R^0$.
\end{remark}

\proofof{Theorem \ref{thm:liftanti}} Let $\sigma \in \mathcal{S}_{\Delta}$ be an involution of the moment polytope and $\varphi^{\sigma}$ its symplectic lift to $M$ given by Lemma~\ref{lem:liftsymp}. We will show that $\varphi^{\sigma}$ and $R^0$ commute and apply Remark~\ref{rk:commuteinvol}. Recall from the proof of Lemma~\ref{lem:liftsymp} that $\varphi^{\sigma}$ is induced by a coordinate permutation $\Phi(z_1,\ldots,z_k) = (z_{\tau(1)},\ldots,z_{\tau(k)})$ on $\C^k$. Similarly, $R^0$ is induced by complex conjugation $\rho^0(z_1,...,z_k) = (\overline{z}_1,...,\overline{z}_k)$. Since the maps $\Phi$ and $\rho^0$ commute, the corresponding maps $\varphi^{\sigma}$ and $R^0$ on $M$ commute as well. Hence we can define the antisymplectic involution 
\begin{equation*}
	R^{\sigma} = R^0 \circ \varphi^{\sigma}.
\end{equation*}
The compatibility condition~$(\ref{eq:comprsigma})$ follows from equations~$(\ref{eq:phisigmacomp})$ and $(\ref{eq:rzerocomp})$.
\proofend

\noindent 
The following alternative view of $R^{\sigma}$ will be used in Section \ref{sec: realdelconst}. Define
\begin{equation}
	\label{eq:rhodef}
	\rho = \rho^0 \circ \Phi : \C^k \rightarrow \C^k, \quad (z_1,\ldots,z_k) \mapsto (\overline{z}_{\tau(1)},\ldots,\overline{z}_{\tau(k)}).
\end{equation}
This is a compatible antisymplectic involution. The corresponding group involution is given by 
\begin{equation}
	\label{eq:rhot}
	\rho_{T^k} = \rho^0_{T^k} \circ \Phi_{T^k} : T^k \rightarrow T^k,   \quad (t_1,\dots,t_k) \mapsto (t^{-1}_{\tau(1)},\dots,t^{-1}_{\tau(k)}).
\end{equation}
Since $\rho^0_{T^k}$ and $\Phi_{T^k}$ both preserve $K$ (see equation~\ref{eq:kpreserved}), so does $\rho_{T^k}$ and hence $\rho$ is $K$-compatible. By Proposition~\ref{prop:symplecticred}, $\rho$ induces an antisymplectic involution $R^{\sigma}$ on $M$. By the compatibility of~$R^{\sigma}$, there is an involutive automorphism~$R^{\sigma}_{T^n}$ on the torus such that~$R^{\sigma}(t.p)=R^{\sigma}_{T^n}(t).R^{\sigma}(p)$. This automorphism is related to~$\rho_{T^k}$ via 
\begin{equation}
	\label{eq:antigroupcomp}
	R^{\sigma}_{T^n}\circ \pi = \pi \circ \rho_{T^k}.
\end{equation}
Note that this is a direct analogue of~$(\ref{eq:groupcomp})$.

\subsection{Classification of compatible maps}
\label{ssec:classification}

The lifts constructed in Sections~$\ref{ssec:liftsymp}$ and~$\ref{ssec:liftantisymp}$ are not unique with respect to their respective compatibility conditions. However, any two maps inducing the same symmetry $\sigma \in \mathcal{S}_{\Delta}$ on the moment polytope are closely related. We start by a lemma which follows from Proposition~\ref{prop:compatibility}. 

\begin{lemma}
\label{lem:equivariantsympl}
Let $\varphi$ be a symplectomorphism on a toric symplectic manifold $(M,\omega,\mu)$ which leaves the moment map invariant, i.e.\ $\mu \circ \varphi = \mu$. Then $\varphi$ is of the form 
	\begin{equation}
	\label{eq:fibrerot}
		p \mapsto \theta(\mu(p)).p
	\end{equation}
for a smooth map\footnote{We define $\theta : \Delta \rightarrow T^n$ to be \emph{smooth} if it comes from a smooth map $\tilde{\theta} : M \rightarrow T^n$ which satisfies $\theta \circ \mu = \tilde{\theta}$.}  $\theta : \Delta \rightarrow T^n$. 
\end{lemma}

\noindent 
Symplectomorphisms of the form \ref{eq:fibrerot} will be denoted by $\psi_{\theta} : M \rightarrow M$. These maps rotate a given fibre $\mu^{-1}(b)$ by an angle $\theta(b) \in T^n$.\\

\proofof{Lemma \ref{lem:equivariantsympl}} 
Let $\varphi$ be a symplectomorphism which leaves the moment map invariant. Hence $\varphi$ preserves the torus fibres, and since $T^n$ acts transitively on each fibre, there is a smooth map $\tilde{\theta} : M \rightarrow T^n$ such that $\varphi$ takes the form
	\begin{equation*}
		\varphi(p) = \tilde{\theta}(p).p, \quad p \in M.
	\end{equation*}
We will prove that $\tilde{\theta}(p)=\tilde{\theta}(p')$ whenever $p$ and $p'$ lie in the same fibre, and thus $\tilde{\theta}$ factors through $\Delta$ to yield a map $\theta : \Delta \rightarrow T^n$. By~\eqref{prop:compatibility}, the symplectomorphism $\varphi$ is $T^n$-equivariant, i.e. 
	\begin{equation*}
		\varphi(t.p) = t.\varphi(p), \quad t \in T^n, p \in M.
	\end{equation*}
Now let $p,p' \in M$ be points in the fibre over $b\in \Delta$. Since $T^n$ acts transitively on $\mu^{-1}(b)$, there is $t \in T^n$ such that $p' = t. p$. Using the $T^n$-equivariance of $\varphi$, we compute
	\begin{equation*}
		\tilde{\theta}(p').p' = \varphi(p') = \varphi(t.p) = t.\varphi(p) = (t\tilde{\theta}(p)).p = (\tilde{\theta}(p) t).p =\tilde{\theta}(p).p'.
	\end{equation*}
Since the action of $T^n$ is free on an open dense subset of $M$, we deduce that $\tilde{\theta}(p)=\tilde{\theta}(p')$ and thus $\tilde{\theta}$ is constant on fibres.
\proofend 

\noindent
Lemma \ref{lem:equivariantsympl} allows us to classify compatible symplectomorphisms as well as compatible antisymplectic involutions. We will use the convention established in \ref{ssec:liftsymp} and \ref{ssec:liftantisymp} and denote the standard lift of $\sigma \in \mathcal{S}_{\Delta}$ by $\varphi^{\sigma}$ and $R^{\sigma}$, respectively.

\begin{proposition}
Let $\varphi$ be a compatible symplectomorphism on $M$ such that $\mu \circ \varphi = \sigma \circ \mu$, for $\sigma \in \mathcal{S}_{\Delta}$. Then there is a smooth map $\theta : \Delta \rightarrow T^n$ such that 
	\begin{equation*}
		\varphi =  \psi_{\theta} \circ \varphi^{\sigma}.
	\end{equation*}
\end{proposition}

\proof By the construction of $\varphi^{\sigma}$, we have $\mu \circ \varphi^{\sigma} = \sigma \circ \mu$ and hence we can apply Lemma~\ref{lem:equivariantsympl} to the symplectomorphism $\varphi \circ (\varphi^{\sigma})^{-1}$ to prove the claim.
\proofend

\begin{proposition}
Let $R$ be a compatible antisymplectic involution on $M$ such that $\mu \circ R = \sigma \circ \mu$, for $\sigma \in \mathcal{S}_{\Delta}$. Then there is a smooth map $\theta : \Delta \rightarrow T^n$ such that 
	\begin{equation}
		\label{eq:rrsigma}
		R =  \psi_{\theta} \circ R^{\sigma}.
	\end{equation}
Furthermore, $\theta$ satisfies $R^{\sigma}_{T^n} (\theta ( \sigma(x))) = \theta(x)^{-1}$ for all $x \in \Delta$.
\end{proposition}

\proof Again, apply Lemma~\ref{lem:equivariantsympl} to the symplectomorphism $R \circ (R^{\sigma})^{-1}$, to prove the first claim. For the condition on $\theta$, we use~\eqref{eq:rrsigma} and compute,
\begin{equation*}
p = R(R(p)) = (\psi_{\theta} \circ R^{\sigma} \circ \psi_{\theta} \circ R^{\sigma}) (p) = (\theta(x) R^{\sigma}_{T^n}(\theta(\sigma(x))).p.
\end{equation*}
Since $T^n$ acts freely on an open dense set in $M$, the claim follows. 
\proofend

\begin{remark}
Let $R_1,R_2 \colon M \rightarrow M$ be two compatible antisymplectic involutions satisfying $\mu \circ R = \sigma \circ \mu$, for some $\sigma \in \mathcal{S}_{\Delta}$. Even though they are related by~$(\ref{eq:rrsigma})$, their respective fixed point sets may have different topology. For example, take $M=\C P^3$ equipped with its standard toric structure and define 
\begin{eqnarray*}
	R_1[z_0:z_1:z_2:z_3] &=& [\bar{z}_1:\bar{z}_0:\bar{z}_3:\bar{z}_2], \\
	R_2[z_0:z_1:z_2:z_3] &=& [-\bar{z}_1:\bar{z}_0:-\bar{z}_3:\bar{z}_2].
\end{eqnarray*}
Both $R_1$ and $R_2$ are compatible with~$\sigma(x,y,z)=(-x-y-z,z,y)$, but~$\Fix R_1 \cong \R P^3$ and~$\Fix R_2 = \varnothing$. See Example~\ref{ex: cp3} and~\cite[Example 2.4]{Kim} for details.
\end{remark}


\section{Real Delzant construction}\label{sec: realdelconst}
In this section, we describe in detail the real Delzant construction stated in Section~\ref{sec: intro}.
Let $(M^{2n},\ow,\mu)$ be a toric symplectic manifold with moment polytope $\Delta=\mu(M)$. 
By the Delzant construction (Section \ref{sec: delzantconst}), we can write
$$
M=\nu^{-1}(0)/K,
$$
where $\nu$ is defined as in \eqref{eq: numap} and $\pi\colon T^k\to T^n$ is the characteristic map with kernel $K=\ker \pi$.  
Let $R=R^{\sigma}$ be the standard antisymplectic involution on $M$ given by the lift of an involution $\sigma\in \mathcal{S}_\Delta$ from Theorem \ref{thm:liftanti}. By Proposition \ref{prop:compatibility} there is a group involution $R_{T^n}$ of $T^n$ satisfying \eqref{eq:comp1}. Recall that
$$
\rho_{T^k}(t_1,\dots,t_k)=(t^{-1}_{\tau(1)},\dots,t^{-1}_{\tau(k)})
$$ is the group involution of $T^k$ defined in equation~\eqref{eq:rhot}. Here $\tau\in S_k$ is the permutation satisfying \eqref{eq:tau}. By Equation~\ref{eq:antigroupcomp}, we see that $\pi\big(\Fix(\rho_{T^k})\big) \subset \Fix(R_{T^n})$ and thus we can define



\begin{definition}
The {\bf characteristic map} $\pi_R$ associated to  $L=\Fix(R)$ is the group homomorphism defined as the restriction 
$$
\pi_R:=\pi|_{\Fix(\rho_{T^k})}\colon \Fix(\rho_{T^k})\to \Fix(R_{T^n}).
$$
\end{definition}
We write $K_R := \ker \pi_R = K\cap \Fix(\rho_{T^k})$ and denote its Lie algebra by $\mathfrak{k}_R$. Recall that
$$
\rho(z_1,\dots,z_k)=(\bar{z}_{\tau(1)},\dots,\bar{z}_{\tau(k)})
$$ is the antisymplectic involution of $\C^k$ given in~\eqref{eq:rhodef}.
We construct the real analogue of the moment map $\nu\colon \C^k\to \mathfrak{k}^*$ defined in \eqref{eq: numap} as follows. We define
$$
\nu_R\colon \Fix(\rho)\to (\mathfrak{k}/\mathfrak{k}_R)^*, \quad \nu_R(z)[\xi]_{\mathfrak{k}_R}:=\langle \nu(z),\xi \rangle \quad \text{for $[\xi]_{\mathfrak{k}_R}\in \mathfrak{k}/\mathfrak{k}_R$}. 	
$$
It follows from Lemma \ref{lem: nuRwelldef} below that $\nu_R$ is well-defined. Recall from Section \ref{ssec:liftantisymp} that $\rho_{T^k}$ preserves $K$ and hence we write
$$
\rho_K:=\rho_{T^k}|_K\colon K\to K
$$
for the group involution on $K$ with $\Fix(\rho_K)=K_R$. One checks that $\Fix((\rho_K)_*)=\mathfrak{k}_R$ and $\Fix(\rho_K^*)=\mathfrak{k}_R^*$. 
\begin{lemma}\label{lem: nuRwelldef}
We have $\langle\nu(z), \xi\rangle=0$ for all $z\in \Fix(\rho)$ and $\xi\in \mathfrak{k}_R$.
\end{lemma}

\begin{proof}
Let $z\in \Fix(\rho)$. Using Proposition \ref{prop:compatibility}, we verify that
\begin{align*}
\nu(z) &= \nu(\rho(z))\\
&= (j^* \circ \nu_0\circ \rho)(z)\\
&= -(j^* \circ \rho_{T^k}^* \circ \nu_0)(z)\\
&= -(\rho_K^* \circ j^*\circ \nu_0)(z)\\
&= -\rho_K^* (\nu (z)).
\end{align*}
Now, for any $\xi\in \mathfrak{k}_R$ we see that
$$
\langle\nu(z), \xi\rangle=-\langle \rho_K^*\big(\nu(z)\big),\xi\rangle=-\langle\nu(z),(\rho_K)_*(\xi)\rangle=-\langle\nu(z),\xi\rangle,
$$
which yields $\langle\nu(z),\xi\rangle=0$.
\end{proof}
We observe that 
$$
\nu^{-1}_R(0) =\nu^{-1} (0) \cap \Fix(\rho)=\Fix(\rho|_{\nu^{-1}(0)}).
$$
Since $\nu_R^{-1}(0)$ is given by the fixed point set of the involution $\rho|_{\nu^{-1}(0)}$ and $\nu^{-1}(0)$ is compact, $\nu_R^{-1}(0)$ is a closed submanifold of $\nu^{-1}(0)$.
\begin{remark}
In the spirit of the Delzant construction, one can also prove that $0\in (\mathfrak{k}/\mathfrak{k}_R)^*$ is a regular value of $\nu_R\colon \Fix(\rho)\to (\mathfrak{k}/\mathfrak{k}_R)^*$. Hence, $\nu_R^{-1}(0)$ is a closed submanifold of $\nu^{-1}(0)$ of dimension $n+\dim K_R$.
\end{remark}
Since the map $\nu$ is $T^k$-invariant and $K$ acts freely on $\nu^{-1}(0)$, the action of $K_R$ on $\nu_R^{-1}(0)$ is well-defined and free. As a result, the quotient $\nu_R^{-1}(0)/K_R$ is a closed manifold. The natural inclusion
$$
\nu_R^{-1}(0)\longhookrightarrow \nu^{-1}(0)
$$
induces the well-defined smooth map on the quotients
$$
I\colon \nu_R^{-1}(0)/K_R\to \nu^{-1}(0)/K,\quad I({K_R}z) = Kz.
$$
Recall that $M=\nu^{-1}(0)/K$.
\begin{lemma}
The map $I\colon \nu_R^{-1}(0)/K_R\to \nu^{-1}(0)/K$ is an embedding.
\end{lemma}
\begin{proof}
Since $\nu_R^{-1}(0)/K_R$ is compact, it suffices to show that $I$ and its differential are injective.

Let $z_1,z_2\in \nu_R^{-1}(0)$ such that $I(K_Rz_1)=I(K_Rz_2)$. Then there is an element $t\in K$ such that $t.z_1=z_2$. Applying the involution $\rho$ to both sides, we get
$$
\rho_K (t). z_1=\rho_K (t).\rho(z_1)=\rho(z_2)=z_2=t.z_1.
$$
Since $K$ acts freely on $\nu^{-1}(0)$, we conclude that $\rho_K(t)=t$ and hence $t\in \Fix(\rho_K)=K_R$. Therefore $K_Rz_1=K_Rz_2$.

We are left with showing that the differential of $I$ is injective. Notice that
$$
T_{z} \big(\nu_{R}^{-1}(0)/K_R\big) \cong T_z \nu^{-1}_{R}(0)/T_z K_Rz
$$
 and 
$$
T_{z} \big(\nu^{-1}(0)/K \big) \cong T_z \nu^{-1}(0)/T_z Kz
$$
 for all $z\in \nu^{-1}_{R}(0)$.
In order to prove that the differential of $I$ is injective, it suffices to prove that~$T_{z} \big(\nu_{R}^{-1}(0)/K_R\big)$ is a subspace of~$T_{z} \big(\nu^{-1}(0)/K \big)$, which follows from
\begin{equation} \label{eq:tangentspace}
T_z K_Rz =T_z Kz \cap T_z \nu^{-1}_R (0).
\end{equation}
To prove this identity, consider the following representations of tangent spaces
\begin{align*}
T_z K_Rz &= \bigg\{X\in T_z \nu^{-1}_R (0)\mid X=\frac{d}{dt} \exp(t\xi).z,\  \xi\in \mathfrak{k}_R\bigg\}\\
&=\bigg\{X\in T_z \nu^{-1}(0) \mid X=\frac{d}{dt} \exp(t\xi).z,\  \xi\in \mathfrak{k}_R \text{ and } \rho_* X=X\bigg\},
\end{align*}
$$
T_z Kz \cap T_z \nu^{-1}_R (0) = \bigg\{X\in T_z \nu^{-1}(0) \mid X=\frac{d}{dt} \exp(t\xi).z,\ \xi\in \mathfrak{k} \text{ and } \rho_* X=X\bigg\}.
$$
{\bf Claim.} \emph{Let $\xi\in \mathfrak{k}$ and $X=\frac{d}{dt}\exp(t\xi).z$. If $\rho_*X=X$, then $(\rho_K)_* \xi=\xi$.} 

By direct computation, we verify that 
$$
\frac{d}{dt}\exp(t\xi).z=\frac{d}{dt}\exp(({\rho_K})_* t\xi).z.
$$
Since the action of $K$ on $\nu^{-1}(0)$ is free, we have that $(\rho_K)_*\xi=\xi$, and so the claim follows.
\end{proof}
We are now in a position to prove the main theorem of the paper.
\begin{proof}[Proof of Theorem \ref{thm: realdelzant}] Since we already know that $I\colon \nu_R^{-1}(0)/K_R\to M$ is an embedding, it suffices to show that $\Fix(R)=I(\nu_R^{-1}(0)/K_R)$. We prove this claim by double inclusion. Recall from Section \ref{ssec:liftantisymp} that $\rho$ is $K$-compatible with $\rho_{T^k}$, whence $R\colon M\to M$ is given by
$$
R(Kz) = K\rho(z)\quad \text{for $z\in \nu^{-1}(0)$}.
$$
Let $K_Rz\in \nu_R^{-1}(0)/K_R$. Since $\rho(z)=z$, we have $R(Kz)=Kz=I(K_Rz)$. This shows that $I(\nu_R^{-1}(0)/K_R)\subset \Fix(R)$, and hence $\Fix(R)$ is not empty. To prove the other inclusion, let $Kz\in \Fix(R)$. Since $R(Kz)=K\rho(z)=Kz$, there exists $t\in K$ such that $\rho(z) = t.z$. We observe that
$$
z=\rho(t.z) = \rho_K(t).\rho(z) = \rho_K(t)t.z.
$$
Since the $K$-action is free, we have $\rho_K(t)=t^{-1}$. Recalling that $K$ is a subtorus, we can choose $\tilde{t}\in K$ such that $\tilde{t}^2=t$ and $\rho_K(\tilde{t})=\tilde{t}^{-1}$. Finally, we get
$$
\rho(\tilde{t}.z) = \rho_K(\tilde{t}).\rho(z) = \rho_K(\tilde{t})t.z = \rho_K(\tilde{t})\tilde{t}^2.z = \tilde{t}.z.
$$
Hence, $Kz=K(\tilde{t}.z)=I(K_R(\tilde{t}.z))$.
This implies that $\Fix(R)\subset I(\nu_R^{-1}(0)/K_R)$, which completes the proof.
\end{proof}

\begin{example} \label{ex: cp3} Consider complex projective space $\C P^3$ with moment map
$$
\mu[z_0:z_1:z_2:z_3] = \frac{4}{\|z\|^2}(|z_1|^2,|z_2|^2,|z_3|^2)-(1,1,1),
$$
where $\|z\|^2=\sum_{j=0}^3|z_j|^2$. We subtract $(1,1,1)$ in order for the normalization $\int_{\C P^3} \mu \omega^3 =0$ to hold. The moment polytope is the 3-simplex given as the convex hull of the vectors
$$ \left\{ (-1,-1,-1),(-1,-1,3),(-1,3,-1),(3,-1,-1) \right\}.  $$ 
Let $\sigma \in \mathcal{S}_{\Delta}$ be the involution defined by $\sigma(x,y,z)=(-x-y-z,z,y)$. Its standard lift is given by the antisymplectic involution
	$$
	R^{\sigma}[z_0:z_1:z_2:z_3]=[\bar{z}_1:\bar{z}_0:\bar{z}_3:\bar{z}_2].
	$$
Using Theorem \ref{thm: realdelzant}, we verify that $\Fix(R^{\sigma})$ is diffeomorphic to $\R P^3$. Figure \ref{fig:cp3} describes the fixed point set of $\sigma$ in the moment polytope $\Delta$ of $\C P^3$.
\begin{figure*}[h]
\begin{center}
\begin{tikzpicture}[scale=1.1]
\draw [fill=blue!8] (0,0)--(0,1)--(1,-0.58)--(0,0);
\draw [fill=blue!4] (0,1)--(1.5,0)--(1,-0.58)--(0,1);
\draw [thick, dashed] (0,0)--(1.5,0);
\draw [thick, red!70] (0.49,-0.29)--(0.75,0.5);
\draw [thick] (0,0)--(0,1)--(1.5,0)--(1,-0.58)--(0,0);
\draw [thick] (0,1)--(1,-0.58);

\begin{scope}[shift={(0.62,0.1)}]
\draw [rotate=-30,dashed] (0,0)--(0.75,0);
\draw [rotate=-30,->] (0.75,0)--(1.8,0);
\draw [dashed] (0,0)--(0.75,0);
\draw [->] (0.75,0)--(2,0);
\draw [dashed] (0, 0)--(0, 0.5);
\draw [->] (0,0.5)--(0,1.5);
\end{scope}

\end{tikzpicture}
\end{center}
\caption{$\Fix(\sigma)$ in $\Delta$.}
\label{fig:cp3}
\end{figure*}
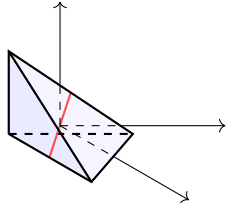
\noindent
We take the primitive outward pointing normal vectors of each facet of $\Delta$,
$$
v_1 = (0, -1, 0),\quad 
	v_2 = (0, 0, -1),\quad 
	v_3 = (-1, 0, 0),\quad 
	v_4 = (1,1,1),
$$
with $\kappa_1=\dots=\kappa_4=1$, and hence
$$
\pi=\begin{pmatrix}
	0 & 0 & -1 & 1 \\
	-1 & 0 & 0 & 1 \\
	0 & -1 & 0 & 1
\end{pmatrix}.
$$
Since $\sigma^*v_1 = v_2$ and $\sigma^*v_3 = v_4$, we obtain two involutions
\begin{eqnarray*}
	\rho_{T^4}(t_1,t_2,t_3,t_4) &=& (t_2^{-1}, t_1^{-1}, t_4^{-1}, t_3^{-1}), \\
	\rho(z_1,z_2,z_3,z_4) &=& (\bar{z}_2, \bar{z}_1, \bar{z}_4, \bar{z}_3).
\end{eqnarray*}
One can direct check that
\begin{eqnarray*}
	 \Fix(\rho_{T^4}) &=& 
	 \{(t,t^{-1},s,s^{-1})\mid t,s\in S^1 \},\\
\Fix(\rho) &=&  \{(z,\bar{z},w,\bar{w}) \mid z,w\in \C\}.	 
\end{eqnarray*}
Since
\begin{eqnarray*}
	K &=& \ker \pi = \{(t,t,t,t)\mid t\in S^1\}\cong S^1,\\
	\nu^{-1}(0) &=& \left\{ (z_1,z_2,z_3,z_4) \mid \sum_{j=1}^4|z_j|^2 = 2 \right\},
\end{eqnarray*}
we obtain
\begin{eqnarray*}
	K_R &=& K\cap \Fix(\rho_{T^4})=\{ (1,1,1,1), (-1,-1,-1,-1) \}\cong \Z_2, \\
	\nu_R^{-1}(0) &=& \nu^{-1}(0)\cap \Fix(\rho)=\left\{(z,\bar{z},w,\bar{w})\mid |z|^2 +|w|^2 = 1\right\}\cong S^3.
\end{eqnarray*}
Since $K_R$ acts by the antipodal action on $\nu_R^{-1}(0)$, we deduce that $\nu_R^{-1}(0)/K_R \cong \R P^3$.
\end{example}

\section{Convexity and Tightness}
In this section we shall prove Theorem \ref{thm: tightconvex}. We follow the same setup as in Section~\ref{sec: realdelconst}. Let $(M,\ow,\mu)$ be a toric symplectic manifold with moment polytope $\Delta=\mu(M)$. By the Delzant construction, we can write $M=\nu^{-1}(0)/K$. Suppose that $R$ is the antisymplectic involution of $M$ which is the lift of an involution $\sigma\in \mathcal{S}_{\Delta}$, see Theorem \ref{thm:liftanti}.

We know that $\Fix(\sigma)=\Delta \cap \{x\in \mathfrak{t}^*\mid \sigma(x)=x\}$ and that the Delzant polytope is convex.  Since the intersection of two convex sets is again convex, so is $\Fix(\sigma)$.
\begin{theorem}\label{thm: convexity}
We have $\Fix(\sigma)=\mu(L)$. In particular, $L$ is nonempty and $\mu(L)$ is convex. 
\end{theorem}

\begin{proof}
Let $x\in \mu (L)$. Then there is an element $Kz\in M=\nu^{-1}(0)/K$ such that $R(Kz) = Kz$ and $\mu(Kz)=x$. Since $R$ is compatible, we obtain
$$\sigma(x)=\sigma(\mu(Kz))=\mu(R(Kz))=\mu(Kz)=x.$$ This imples that $\mu(L)\subset\Fix(\sigma)$. 

Let $x\in \Fix(\sigma)\subset\Delta$ and let $Kz \in M$ with $\mu(Kz)=x$. We show that there is~$\tilde{t} \in T^n$ such that $\tilde{t}.Kz \in L$. Note that
$$\mu(K\rho(z))=\mu(R(Kz))=\sigma(\mu(Kz))=\sigma(x)=x=\mu(Kz),$$
and hence there is an element $t\in T^n$ such that $K\rho(z)=t.Kz$. This follows from the fact that $T^n$ acts transitively on fibres. Applying the involution $R$, we obtain
\begin{equation}\label{equ : conv}
Kz = R_{T^n} (t)t.Kz.
\end{equation}
{\bf Claim.} \emph{There exists $\tilde{t}\in T^n$ such that $\tilde{t}^2=t$ and $Kz=R_{T^n}(\tilde{t})\tilde{t}.Kz$.}

The claim is obvious in case $t=1$ and thus we assume that $t \neq 1$. Denote by $S^1\langle t\rangle$ the subgroup of $T^n$ generated by $t\in T^n$ and consider the group homomorphism
$$
\phi\colon S^1\langle t \rangle \to T^n, \quad \phi(s)=R_{T^n}(s)s.
$$
Since the stabilizer $\Stab(Kz)$ of the $T^n$-action at the point $Kz\in M$ is a subtorus and, by~\eqref{equ : conv}, $\phi(t)=R_{T^n}(t)t\in \Stab(Kz)$, we have $\im \phi \le \Stab(Kz)$. If we choose $\tilde{t}\in S^1\langle t\rangle$ such that $\tilde{t}^2=t$, then $\phi(\tilde{t})=R_{T^n}(\tilde{t})\tilde{t}\in \Stab(Kz)$. Hence, the claim follows.

In order to show $\tilde{t}.Kz\in L= \Fix(R)$, we verify
$$
R(\tilde{t}.Kz)=R_{T^n}(\tilde{t}).R(Kz)=R_{T^n}(\tilde{t})t.Kz=R_{T^n}(\tilde{t})\tilde{t}^2.Kz=\tilde{t}.Kz.
$$
This completes the proof.
\end{proof}
We denote the set of critical points of $f\in C^{\infty}(M)$ by $\Crit(f)$. We recall Duistermaat's tightness theorem \cite[Theorem 3.1]{Duist}.
\begin{theorem}[Duistermaat]\label{thm: duistermaat}
	 Let $(M,\omega,\mu)$ be a compact connected Hamiltonian $T^n$-space. Suppose that $R$ is an antisymplectic involution on $M$ such that $\mu\circ R=\mu$ and $L=\Fix(R)$ is nonempty. For any $\xi\in \mathfrak{t}$ we have 
	 $$\dim H_{*} (L;\Z_2)= \dim H_{*} (\Crit(H_{\xi}|_L);\Z_2).$$ 
\end{theorem}
Our tightness result is a corollary of this theorem.
\begin{proof}[Proof of Theorem \ref{thm: tightconvex}]
The convexity result follows from Theorem \ref{thm: convexity}. To prove the tightness, let $\sigma^*$ denote the transpose of $\sigma \in \Aut_{\Z}(\mathfrak{t}^n)^*$. Then $\sigma^*$ is an involution on $\mathfrak{t}^n$ and for any $\xi \in \mathfrak{t}^n$ we can decompose $\xi$ as $\xi=\xi_1 +\xi_2$ with $\sigma^* \xi_1 = -\xi_1$ and $\sigma^* \xi_2 =\xi_2$. Note that  $H_{\xi}=H_{\xi_1}+H_{\xi_2}$. If $x\in \Fix(R)$, we have
$$\langle \mu(x), \xi_1 \rangle = \langle \mu(R(x)), \xi_1 \rangle 
= \langle \sigma(\mu(x)), \xi_1 \rangle 
= \langle \mu(x), \sigma^* \xi_1 \rangle   
= \langle \mu(x), -\xi_1 \rangle,
$$
which implies
\begin{equation}\label{eq: xi1vanish}
H_{\xi_1}|_L\equiv 0.
\end{equation}
Furthermore, we see that
$$
\langle \mu \circ R, \xi_2 \rangle = \langle \sigma \circ \mu, \xi_2 \rangle = \langle \mu, \sigma^* \xi_2 \rangle = \langle \mu, \xi_2 \rangle.
$$
We take the subtorus $T_0 \stackrel{j} \hookrightarrow T^n$ such that $ \Lie(T_0) = \mathfrak{t}_0 = \{\xi\in \mathfrak{t}^n \mid \sigma^* \xi =\xi\}$, i.e., $T_0$ is the identity component of $\Fix(\sigma^*)$. Then the induced $T_0$-action on $M$ is Hamiltonian and has moment map $\tilde{\mu}:=j^*\mu$ with $\tilde{\mu} \circ R = \tilde{\mu}$. Since $\xi_2\in \mathfrak{t}_0$ by Theorem \ref{thm: duistermaat} and \eqref{eq: xi1vanish} we obtain
\begin{eqnarray*}
\dim H_{*} (L;\Z_2) &=& \dim H_{*} (\Crit(H_{\xi_2}|_L);\Z_2) \\
&=&  \dim H_{*} (\Crit(H_{\xi}|_L);\Z_2).	
\end{eqnarray*}
This completes the proof.
\end{proof}
The following example illustrates that the tightness and convexity do not hold if we drop the compatible condition on the real Lagrangian $L=\Fix(R)$.
\begin{example}\label{ex: tightconvexfails}
Consider the two-sphere $S^2$ equipped with the Euclidean area form. Any embedded loop in $S^2$ dividing $S^2$ into two discs of equal area is a real Lagrangian. Pick $\xi=1\in \mathfrak{t}\cong \R$ so that $\mu=H_\xi$. It is not difficult to find a real Lagrangian $L$ in $S^2$ such that $\Crit(H_\xi|_L)$ consists of four critical points, see Figure~\ref{fig: tightnessfails}. Hence, tightness fails for the real Lagrangian $L\cong S^1$, namely,
$$\dim H_{*} (L;\Z_2)= 2 \neq 4 =\dim H_{*} (\Crit(H_{\xi}|_{L});\Z_2).$$	
\end{example}
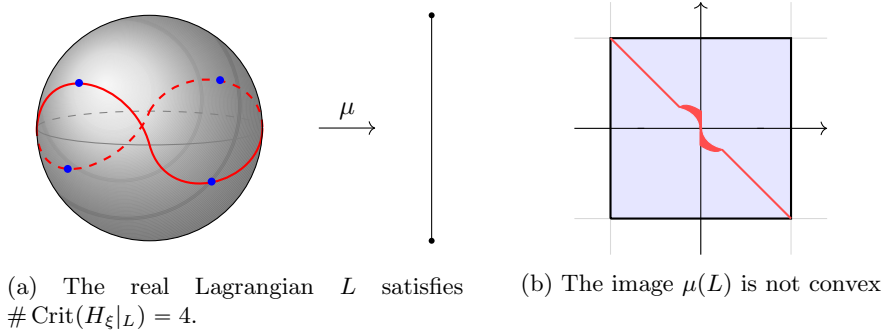
\begin{figure}[h]
\begin{subfigure}{0.48\textwidth}
   \centering
\begin{tikzpicture}[scale = 0.75]
    \shade[ball color = gray!40, opacity = 0.4] (0,0) circle (2cm);
    \draw (5,-2)--(5,2);
    \fill[fill=black] (5,-2) circle (1.5pt);
    \fill[fill=black] (5,2) circle (1.5pt);
    \draw[gray] (0,0) +(180:2) arc (180:360:2 and 0.3);
    \draw[gray, dashed] (0,0) +(0:2) arc (0:180:2 and 0.3);
    \draw[red, thick] (-2,0) .. controls +(85:1.3) and +(105:1.2) .. (0,-0.3);
    \draw[red, thick] (0,-0.3) .. controls +(285:1.2) and +(270:1) .. (2,0);
    \draw[red, thick, dashed] (-2,0) .. controls +(270:1.3) and +(250:1) .. (0,0.3);
    \draw[red, thick, dashed] (0,0.3) .. controls +(60:1) and +(90:1) .. (2,0);
    \draw (0,0) circle (2cm);
    \fill[fill=blue] (-1.45,-0.725) circle (2pt);
    \fill[fill=blue] (-1.25,0.8) circle (2pt);
    \fill[fill=blue] (1.1,-0.95) circle (2pt);
    \fill[fill=blue] (1.25,0.85) circle (2pt);
    \draw[->] (3,0)--node[above]{$\mu$}(4,0);
    \draw [white] (0, -2.3)--(1, -2.3);
        \draw [white] (0, 2.3)--(1, 2.8);
\end{tikzpicture}
\caption{The real Lagrangian $L$ satisfies $\#\Crit(H_\xi|_L)=4$.}
\label{fig: tightnessfails}
\end{subfigure}
\begin{subfigure}{0.48\textwidth}
   \centering
	\begin{tikzpicture}[scale=1.2]
\draw[step=1.0,black!15,thin] (-1.4,-1.4) grid (1.4,1.4);
\draw [thick,fill=blue!10] (-1,-1)--(-1,1)--(1,1)--(1,-1)--(-1,-1);

\draw[thick, red!70] (-1,1)--(1,-1);
\draw [->] (-1.4,0)--(1.4,0);
\draw [->] (0,-1.4)--(0,1.4);


\draw [rotate=-45, fill=red!70, red!70] (0,0) ellipse (0.33cm and 0.15cm);
\draw [fill=blue!10, blue!10](-0.32,0) ellipse (0.3cm and 0.23cm);

\draw [fill=blue!10, blue!10](0.32,0) ellipse (0.3cm and 0.23cm);
\draw [fill=blue!10, blue!10] (0.02, 0)--(0.5, 0)--(0.5, 0.5)--(0.02, 0.5)--(0.02, 0);
\draw [fill=blue!10, blue!10] (-0.02, 0)--(-0.5, 0)--(-0.5, -0.5)--(-0.02, -0.5)--(-0.02, 0);
\draw (-0.7, 0)--(0.7,0);
\draw [red!70, fill=red!70] (-0.008, -0.02)--(0.01, -0.02)--(0.008, 0.02)--(-0.01, 0.02)--(-0.008, -0.02);

\end{tikzpicture}
\caption{The image $\mu(L)$ is not convex}
\label{fig: convexityfails}
\end{subfigure}
\caption{ Examples in which convexity or tightness fail.}
\end{figure}
One can easily find a real Lagrangian torus $L$ in $S^2\times S^2$ such that $\mu(L)$ is not convex. Indeed, let
$$
L'=\Fix(R')=\{(x,-x) \mid x\in S^2\}
$$ be the real Lagrangian in $S^2\times S^2$, where $R'(x,y)=(-y,-x)$. Note that $\mu(L')=\{(\xi,-\xi) \mid \xi\in \mathfrak{t}^* \}\cap \Box$ with $\Box=\mu(S^2\times S^2)$. Then one can choose a suitable Hamiltonian diffeomorphism $\phi$ on $S^2\times S^2$ which is compactly supported in $\mu^{-1}(U_0)$, where $U_0\subset \Box$ is a small open set of the origin, such that the antisymplectic involution $R:=\phi  \circ R'\circ \phi^{-1}$ has fixed point set $\Fix(R)=\phi(\Fix(R'))$ whose moment image is wiggled near the origin. See Figure \ref{fig: convexityfails}.

\section{Real Lagrangians in toric symplectic del Pezzo surfaces}\label{sec: delpezzo}
As an application of our real Delzant construction, we study real Lagrangians in toric symplectic del Pezzo surfaces. Recall that a \emph{symplectic del Pezzo surface} is one the following symplectic 4-manifolds:
\begin{enumerate}
	\item $Q:=S^2 \times S^2$ the product of the 2-sphere $(S^2,\ow)$, where $\ow$ denotes an area form on $S^2$,
	\item $X_k:=\C P^2\# k\overline{\C P^2}$ the $k$-fold monotone symplectic blow-up of $\C P^2$ for $0\le k\le 8$.
\end{enumerate}
Every closed monotone symplectic 4-manifold is one of the symplectic del Pezzo surfaces and that the monotone symplectic structures on del Pezzo surfaces are unique, see \cite[Section 1]{Vianna} for the references.
 The following is an analogue of \cite[Lemma 2.3]{Evans}, which gives a homological obstruction for being Lagrangian in symplectic del Pezzo surfaces.
\begin{lemma}\label{lem: lagindelpezzo}
Symplectic del Pezzo surfaces contain no Lagrangian $\Sigma_g$ for all $g\ge 2$, where $\Sigma_g$ denotes the closed oriented surface of genus $g$. Furthermore, $X_1$ does not contain any Lagrangian sphere. 
\end{lemma}
\begin{proof}
Let $L$ be an orientable Lagrangian in a symplectic del Pezzo surface $X$. It suffices to show that $\chi(L)\ge 0$. Since $L$ is Lagrangian and $X$ is monotone, the homology class $[L]\in H_2(M;\Z)$ satisfies 
\begin{equation}\label{eq:cL}
c_1(X)[L]=0,\quad [L]\cdot[L]=-\chi(L).
\end{equation}
The second property follows from the Weinstein neighborhood theorem, which asserts that the normal bundle of $L$ is isomorphic to $T^*L$, and from $\dim L=2$.\\\\
{\bf Case of $Q$.} The first Chern class $c_1(Q)$ is Poincar\'{e} dual to $2\alpha+2\beta$, where $\alpha=[S^2\times \{pt\}]$ and $\beta=[\{pt\}\times S^2]$ are generators of $H_2(Q;\Z)$. Let $[L]=a\alpha+b\beta$ for $a,b\in \Z$. Then the identities \eqref{eq:cL} become
$$
2a+2b=0,\quad 2ab=-\chi(L),
$$
which shows that $\chi(L)=2 b^2\ge 0$.\\\\
{\bf Case of $X_k$ for $1\le k\le 8$.} Note that $c_1(X_k)$ is Poincar\'{e} dual to the class
$$
3H-\sum_{j=1}^kE_j\in H_2(X_k;\Z),
$$
where $H=[\C P^1]$ and the $E_j$ are the classes of the exceptional spheres. Write $[L]=aH-\sum_{j=1}^kb_jE_j$. Equation \eqref{eq:cL} becomes
\begin{eqnarray*}
	3a-\sum_{j=1}^k b_j=0,\quad a^2-\sum_{j=1}^kb_j^2=-\chi(L),
\end{eqnarray*}
which yields
$$
9\sum_{j=1}^kb_j^2-\left(\sum_{j=1}^k b_j\right)^2=9\cdot \chi(L).	
$$
This identity can be rewritten as 
$$
9\chi(L)=(9-k)\sum_{j=1}^kb_j^2+\sum_{i<j}(b_i-b_j)^2.
$$
Since $k \leq 8$ we conclude that  $\chi (L) \geq 0$
also in this case.

In order to prove the last statement, note that if $L$ were a Lagrangian sphere in $X_1$ with $[L]=aH - bE_1$, then by the second property in equation~$(\ref{eq:cL})$ we would have $a^2 - b^2 = -2$.
\end{proof}
Smith theory \cite[Theorems 4.1 and 4.3, Chapter III]{Bredon} implies that any real Lagrangian $L$ in a symplectic manifold $(M,\ow)$ satisfies
\begin{eqnarray*}
	\dim H_*(M; \Z_2) &\ge& \dim H_*(L;\Z_2),\\
	\chi(M) &=& \chi(L)\mod 2.
\end{eqnarray*}
Together with Lemma \ref{lem: lagindelpezzo} one obtains Table \ref{tab: delpezzo} for the candidates of diffeomorphism types except for the cases of $\R P^2\#\R P^2$ in $S^2\times S^2$ and $T^2$ in $X_1$. We are only interested in the symplectic del Pezzo surfaces that have a toric structure, namely $S^2\times S^2$ and $X_k$ for $0\le k \le 3$. By \cite[Lemma 4.4]{Kim}, we can exclude $\R P^2\# \R P^2$ in $S^2\times S^2$ in Table \ref{tab: delpezzo}. 

\subsection{Arnold lemma and its application}
In order to show that $T^2$ cannot be a real Lagrangian in $X_1$, we employ the Arnold lemma which we now explain. We refer to \cite{Arnold} for details. 
Let $\tau$ be an orientation-preserving involution of a closed oriented manifold~$X^{4}$. Assume that the fixed point set~$\Fix(\tau)$ of~$\tau$ is a closed surface. The involution $\tau$ induces the isomorphism $\tau_*\colon H_2(X;\Z_2)\to H_2(X;\Z_2)$. We define the symmetric $\Z_2$-bilinear form $\Phi_\tau$ (called \emph{the twisted intersection form}) on $H_2(X;\Z_2)$ by
$$
\Phi_\tau(\alpha,\beta)=\alpha\cdot \tau_*(\beta)\mod 2,
$$
where $\cdot$ denotes the intersection number. Recall that $w\in H_2(X;\Z_2)$ is called \emph{characteristic class (or fundamental class)} of $\Phi_\tau$ if $\Phi_\tau(w,\alpha)=\Phi_\tau(\alpha,\alpha)$ for all~$\alpha\in H_2(X;\Z_2)$. Since $\Phi_\tau$ is non-degenerate, there exists a unique characteristic class of $\Phi_\tau$. Note that the characteristic class of $\Phi_\tau$ vanishes if and only if $\Phi_\tau(\alpha,\alpha)=0$ for all $\alpha\in H_2(X;\Z_2)$. The following is the so-called \emph{Arnold lemma}, see \cite[Lemma 3]{Arnold} for the proof.
\begin{lemma}\label{lem: arnold}
	The $\Z_2$-homology class $[\Fix(\tau)]_{\Z_2}\in H_2(X;\Z_2)$ represented by $\Fix(\tau)$ is the characteristic class of $\Phi_\tau$.
\end{lemma}
We are ready to prove the following lemma.
\begin{lemma}
Assume that $L$ is a real Lagrangian in $X_1$ that is diffeomorphic to a closed connected surface. Then $L$ must be non-orientable.
\end{lemma}
\begin{proof}
Let $L=\Fix(R)$ for some antisymplectic involution $R$ of $X_1$. Assume to the contrary that $L$ is orientable, and hence $L$ represents a $\Z$-homology class   $[L]_{\Z}\in H_2(X_1;\Z)$. Using the notations in the proof of Lemma \ref{lem: lagindelpezzo}, we take the generators $H$ and $E_1$ on $H_2(X_1;\Z)$. Since $R$ is orientation-preserving, $R_*$ preserves the intersection form. Using also that $R_*^2 = \id$
and that $R^*[\omega] = -[\omega]$ on $H^2(X_1;\Z)$, one computes that that the induced map $R_*$ on $H_2(X_1;\Z)$ is given by $R_*=-\id$.
Since $R_*[L]_{\Z}=[L]_{\Z}$, we obtain $[L]_{\Z}=0$ and hence $[L]_{\Z_2}=0$ as well. Noting that $R_*^{\Z_2}=\id$ on $H_2(X_1;\Z_2)$, the twisted intersection form $\Phi_R$ is the usual mod 2 intersection form of $X_1$. By Lemma \ref{lem: arnold}, the characteristic class of $\Phi_R$ vanishes and so the intersection form of $X_1$ must be even, which yields a contradiction.
\end{proof}
We conclude that there are no real Lagrangian tori in $X_1$.

\subsection{Constructions of explicit real Lagrangians}
We now explicitly construct real Lagrangians that realize all diffeomorphism types in Table~\ref{tab: delpezzo}.
\begin{remark}
On every toric manifold $M$, we can lift the trivial involution $\sigma_0 = \id$, which yields the natural antisymplectic involution $R^0$ on $M$. Its fixed point set corresponds to the \emph{real locus}, see Proposition~\ref{prop:stdantiinvol} and Remark~\ref{rk:reallocus}. In particular, the real locus of $X_k$, that is diffeomorphic to $\#_{k+1}\R P^2$, is the real Lagrangian $\Fix(R^{\id})$ in $X_k$ for each $0\le k \le 3$. 

\end{remark}
\begin{example}\label{example: cp2}
Consider $\C P^2$ equipped with the Fubini-Study form $\ow_{\FS}$. Its Delzant polytope $\Delta$ is defined by the outward pointing normal vectors
$$v_1=(-1,0),\quad v_2=(0,-1), \quad v_3=(1,1).$$
There is one non-trivial involution in $\mathcal{S}_\Delta$, namely the reflection $\sigma(x,y)=(y,x)$ with respect to the diagonal line, see Figure \ref{fig: 1}.
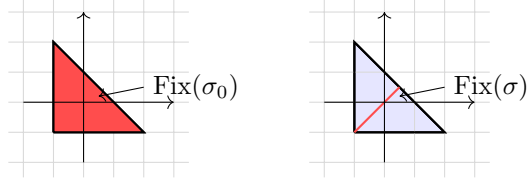
\begin{figure*}[h]
\begin{center}
\begin{tikzpicture}[scale=0.4]

\begin{scope}[xshift=-10cm]
\begin{scope}[xshift=-1cm, yshift=-1cm]

\draw[step=1.0,black!15,thin] (-1.5,-1.5) grid (4.5,4.5);

\draw [thick] (0,0)--(0,3)--(3,0)--(0,0);
\draw [thick,fill=red!70] (0,0)--(0,3)--(3,0)--(0,0);
\draw [->] (3, 1.5)--(1.5,1.2);

\node at (3,1.5)[right]{$\Fix(\sigma_0)$};
\end{scope}
\draw [->] (-2,0)--(3,0);
\draw [->] (0,-2)--(0,3);
\end{scope}

\begin{scope}[xshift=-1cm, yshift=-1cm]
\draw [thick,fill=blue!10] (0,0)--(0,3)--(3,0)--(0,0);

\draw[step=1.0,black!15,thin] (-1.5,-1.5) grid (4.5,4.5);

\draw [thick] (0,0)--(0,3)--(3,0)--(0,0);
\draw [thick,red!70] (0,0)--(1.5,1.5);
\draw [->] (3,1.5)--(1.5,1.2);

\node at (3,1.5)[right]{$\Fix(\sigma)$};
\end{scope}
\draw [->] (-2,0)--(3,0);
\draw [->] (0,-2)--(0,3);
\end{tikzpicture}		
\end{center}
\caption{Involutions on the 2-simplex $\Delta$}
\label{fig: 1}
\end{figure*}

\noindent
We use the real Delzant construction to prove that $\Fix(R^{\sigma})$ is diffeomorphic to $\R P^2$. Since $\sigma$ exchanges the normal vetors $v_1$ and $v_2$, we obtain
\begin{eqnarray*}
	&&\rho\colon \C^3\to \C^3, \quad \rho(z_1,z_2,z_3)=(\bar{z}_2,\bar{z}_1,\bar{z}_3), \\
	&&\rho_{T^3}\colon T^3\to T^3, \quad \rho_{T^3}(t_1,t_2,t_3)=(t_2^{-1},t_1^{-1},t_3^{-1}).
\end{eqnarray*}
Observe that
\begin{eqnarray*}
&&\Fix(\rho)=\{(z,\bar{z},x) \mid z\in \C,\ x\in \R \},\\
&& \Fix(\rho_{T^3})=	\left\{(t,t^{-1},s) \mid t\in S^1,\ s=\{1,-1 \} \right\} \cong S^1\oplus \Z_2.
\end{eqnarray*}
Hence, the kernel of $\pi_R$ is
$$
K_R=K\cap \Fix(\rho_{T^3})=\{(1,1,1),(-1,-1,-1)\}\cong \Z_2
$$
and
$$
\nu_R(z,\bar{z},x)=|z|^2+\frac{x^2}{2}-3.
$$
Therefore $\nu_R^{-1}(0)$ is a 2-sphere on which $K_R$ acts by the antipodal mapping and hence
$$
\Fix(R^{\sigma})\cong \nu_R^{-1}(0)/K_R\cong \R P^2.
$$
In fact, it follows from Smith theory that any real Lagrangian in $\C P^2$ (not necessarily compatible with the torus action) is diffeomorphic to $\R P^2$.
\end{example}

\begin{example}
Consider $(S^2 \times S^2, \omega_0 \oplus \omega_0, \mu)$, where $\omega_0$ is the area form on the sphere and its Delzant polytope is
$$
\Box:=[-1,1]^2=\mu(S^2\times S^2)
$$
with outward pointing normal vectors
$$
v_1=(1,0),\quad v_2=(0,1),\quad v_3=(-1,0),\quad v_4=(0,-1)
$$
and $\kappa_1=\dots=\kappa_4=1$. Hence,
$$
\pi=\begin{pmatrix}
	1 & 0 & -1 & 0 \\
	0 & 1 & 0 & -1
\end{pmatrix},\quad K = \{(e^{2\pi i \alpha},e^{2\pi i \beta},e^{2\pi i \alpha},e^{2\pi i \beta})\} 
$$
We consider the four involutions in $\mathcal{S}_{\Box}$ given in Figure~\ref{fig: involutionsonS2S2}, namely
$$
\sigma_0 = \begin{pmatrix}
	1 & 0 \\
	0 & 1
\end{pmatrix}, \quad 
\sigma_1 = \begin{pmatrix}
	-1 & 0 \\
	0 & -1
\end{pmatrix},\quad 
\sigma_2 = \begin{pmatrix}
	0 & -1 \\
	-1 & 0
\end{pmatrix},\quad 
\sigma_3 =\begin{pmatrix}
	-1 & 0 \\
	0 & 1
\end{pmatrix}.
$$
\begin{figure}[h]
\begin{subfigure}{0.24\textwidth}
   \centering
\begin{tikzpicture}[scale=0.6]
\draw[step=1.0,black!15,thin] (-2.5,-2.5) grid (2.5,2.5);
\draw [thick,fill=red!70] (-1,-1)--(-1,1)--(1,1)--(1,-1)--(-1,-1);
\draw [->] (-2,0)--(2,0);
\draw [->] (0,-2)--(0,2);

\node at (1,1.5)[right]{$\sigma_0$};
\end{tikzpicture}
  \caption{$\R P^1\times \R P^1\cong T^2$}
  \label{fig: center}
\end{subfigure}
\begin{subfigure}{0.24\textwidth}
   \centering
\begin{tikzpicture}[scale=0.6]
\draw[step=1.0,black!15,thin] (-2.5,-2.5) grid (2.5,2.5);
\draw [thick,fill=blue!10] (-1,-1)--(-1,1)--(1,1)--(1,-1)--(-1,-1);
\draw [->] (-2,0)--(2,0);
\draw [->] (0,-2)--(0,2);
\fill[thick, red!70] (0,0)  circle[radius=2pt];
\node at (1,1.5)[right]{$\sigma_1$};
\end{tikzpicture}
  \caption{$\mathbb{T}_{\Cl}=T^2$}
  \label{fig: center}
\end{subfigure}
\begin{subfigure}{0.24\textwidth}
  \centering
\begin{tikzpicture}[scale=0.6]
\draw[step=1.0,black!15,thin] (-2.5,-2.5) grid (2.5,2.5);
\draw [thick,fill=blue!10] (-1,-1)--(-1,1)--(1,1)--(1,-1)--(-1,-1);
\draw [->] (-2,0)--(2,0);
\draw [->] (0,-2)--(0,2);
\draw[thick, red] (-1,1)--(1,-1);

\node at (1,1.5)[right]{$\sigma_2$};
\end{tikzpicture}
  \caption{$\overline{\Delta}=S^2$}
 \label{fig: antidiagonal}
\end{subfigure}
\begin{subfigure}{0.24\textwidth}
  \centering
\begin{tikzpicture}[scale=0.6]
\draw[step=1.0,black!15,thin] (-2.5,-2.5) grid (2.5,2.5);
\draw [thick,fill=blue!10] (-1,-1)--(-1,1)--(1,1)--(1,-1)--(-1,-1);
\draw [->] (-2,0)--(2,0);
\draw [->] (0,-2)--(0,2);
\draw[thick, red] (0,-1)--(0,1);

\node at (1,1.5)[right]{$\sigma_3$};
\end{tikzpicture}
  \caption{$T^2$}
\end{subfigure}
\caption{Involutions of $\Box=[-1,1]^2$.}
 \label{fig: involutionsonS2S2}
\end{figure}
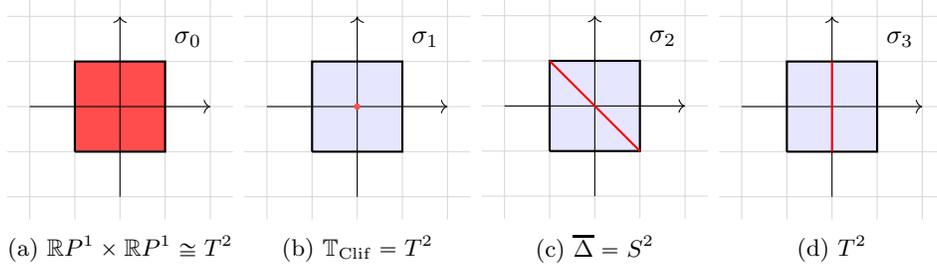
\\
{\it Involution $\sigma_0$.} Since the real locus of $S^2\times S^2=\C P^1\times \C P^1$ is diffeomorphic to $\R P^1\times \R P^1 \cong T^2$, so is $\Fix(R^{\sigma_0})$.\\\\
{\it Involution $\sigma_1$.} Since $\Fix(\sigma_1)$ is a singleton, the corresponding real Lagrangian is given by the Lagrangian fibre $\Fix(R^{\sigma_1})=\mu^{-1}\big((0,0)\big)\cong T^2$, which is called the \emph{Clifford torus} in $S^2\times S^2$.\\\\
{\it Involution $\sigma_2$.} We observe that
$$
\nu^{-1}(0)=\big\{(z_1,z_2,z_3,z_4)\mid |z_1|^2 +|z_3|^2 = 4,\ |z_2|^2 +|z_4|^2 = 4\big\}\cong S^3 \times S^3.
$$
Since $\sigma_2$ exchanges $v_1$ with $v_4$ and $v_2$ with $v_3$, we obtain
\begin{eqnarray*}
&& \Fix(\rho)=\{(z,w,\bar{w},\bar{z}) \mid z,w \in \C\},\\
&& \Fix(\rho_{T^3})=\{(t,s,s^{-1},t^{-1})\mid t,s\in S^1\}.
\end{eqnarray*}
Hence, 
$$
\nu_R^{-1}(0)=\{(z,w,\bar{w},\bar{z})\mid |z|^2+|w|^2=4\}\cong S^3
$$ 
and 
$K_R=K\cap \Fix(\rho_{T^3})=\{(t,t^{-1},t,t^{-1})\mid t\in S^1\}$. Hence the $K_R$-action on $\nu_R^{-1}(0)$ can be identified with the Hopf action on $S^3$, and we obtain
$$
\Fix(R^{\sigma_2})\cong \nu_R^{-1}(0)/K_R\cong S^2,
$$
which is called the \emph{antidiagonal sphere} $\Fix(R^{\sigma_2})=\overline{\Delta}:=\{(x,-x)\mid x\in S^2\}$.\\\\
{\it Involution $\sigma_3$.} Similarly, we have
\begin{eqnarray*}
	&&\nu_R^{-1}(0) = \{(z,x_1,\bar{z},x_2) \mid z\in \C,\ x_1,x_2\in \R,\ |z|^2=2,\ x_1^2+x_2^2=2\}\cong T^2,  \\
	&& K_R=\left\{(s_1,s_2,s_1,s_2)\mid s_1,s_2 \in \{1,-1\} \right\}\cong \Z_2^2,
\end{eqnarray*}
and hence 
$$
\Fix(R^{\sigma_3})\cong \R P^1\times \R P^1\cong T^2.
$$
Recall that any two embedded loops in $S^2$ are Hamiltonian isotopic if they divide the sphere into two discs with equal area. Using this, one can easily show that the real Lagrangian tori $\Fix(R^{\sigma_0})$, $\Fix(R^{\sigma_1})$, and $\Fix(R^{\sigma_3})$ are (pairwise) Hamiltonian isotopic to each other.
\end{example}

\begin{example}
Consider the monotone toric symplectic manifold $\C P^2 \# \overline{\C P^2}$ with moment polytope the isosceles trapezoid $\Delta$ depicted in Figure \ref{fig: blow up}. Then we have
$$
v_1 =(-1,0),\quad v_2 =(0,-1),\quad v_3 =(1,1),\quad v_4 =(-1,-1)
$$
and  $\kappa_1 =\dots= \kappa_4 =1$. Note that
$$
\pi=\begin{pmatrix}
	-1 & 0 & 1 & -1 \\
	0 & -1 & 1 & -1
\end{pmatrix},\quad K=\{(e^{2\pi i \alpha},e^{2\pi i \alpha},e^{2\pi i (\alpha + \beta)},e^{2\pi i \beta})\}, 
$$
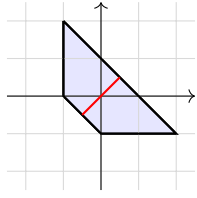
\begin{figure*}[h]
\begin{center}
\begin{tikzpicture}[scale=0.5]
\draw [thick,fill=blue!10] (-1,2)--(2,-1)--(0,-1)--(-1,0)--(-1,2);

\draw[step=1.0,black!15,thin] (-2.5,-2.5) grid (2.5, 2.5);
\draw [thick] (-1,2)--(2,-1)--(0,-1)--(-1,0)--(-1,2);

\draw [->] (-2.5,0)--(2.5,0);
\draw [->] (0,-2.5)--(0,2.5);
\draw [thick,red] (-0.5, -0.5)--(0.5, 0.5);
\end{tikzpicture}
\end{center}
\caption{Monotone blow-up of $\C P^2$}
\label{fig: blow up}
\end{figure*}
\noindent 
There is only one non-trivial involution on $\Delta$, namely $\sigma(x,y)=(y,x)$. We show that $\Fix(R^\sigma)\cong \R P^2\# \R P^2$ using the real Delzant construction. Observe that
$$
\nu^{-1}(0)=\Big\{(z_1,z_2,z_3,z_4)\ \Big|\ |z_1|^2 +|z_2|^2 + |z_3|^2  = 6,\ |z_3|^2+|z_4|^2=4 \Big\}.
$$
The involution $\sigma$ acts on $\Delta$ by exchanging $v_1$ with $v_2$ and leaving $v_3$ and $v_4$ invariant, whence
\begin{eqnarray*}
	&& \Fix(\rho) =\{
	(z,\bar{z},x,y)\mid z\in \C,\ x,y\in \R\}, \\
	&& \Fix(\rho_{T^4}) = \left \{ (t,t^{-1},s_1,s_2)\mid t\in S^1,\ s_1,s_2 \in \{1,-1\} \right\}\cong S^1\oplus \Z_2^2.
\end{eqnarray*}
We obtain
\begin{eqnarray*}
&&	\nu_R^{-1}(0) =\{(z,\bar{z}, x,y) \mid 2|z|^2+x^2=6,\ x^2+y^2=4\}\cong T^2,\\
&& K_R= \left\{(1,1,1,1),(-1,-1,-1,1),(1,1,-1,-1), (-1,-1,1,-1) \right\}\cong \Z_2 \oplus \Z_2.
\end{eqnarray*} 
We claim that the quotient map
$$
T^2 \cong \nu_R^{-1}(0) \longrightarrow \nu_R^{-1}(0)/K_R
$$
is a 4-fold covering of the Klein bottle $\R P^2\# \R P^2$. This follows from Table~\ref{tab: delpezzo}. To see this explicitly, we first identify $\nu_R^{-1}(0)$ with the torus $T^2$ obtained by the product of two circles, namely
$$
S_{xy}:=\{x^2+y^2=4\} \quad\text{and}\quad S_{z}:=\{2|z|^2+x^2=6\}.
$$
Note that $S_z$ varies depending on $x$.
We obtain the identification in Table~\ref{tab: kleinbottle}. Using this, we see that the quotient map above is a 4-fold covering of $\R P^2\#\R P^2$ as desired, see Figure \ref{fig: 4foldcover}.

\begin{table}[hbt]
  \begin{tabular}{c|c|c|c}

  $K_R$ &&$S_z$ &  $S_{xy}$\\
  \hline
  \rule{0pt}{3ex}    
$(1,1,1,1)$ & &$\id$ & $\id$ \\  \rule{0pt}{3ex}    
$(-1,-1,-1,1)$ &  & antipodal map & $y$-axis reflection\\  \rule{0pt}{3ex}    
$(1,1,-1,-1)$ & & $\id$ & antipodal map \\  \rule{0pt}{3ex}    
$(-1,-1,1,-1)$ & & antipodal map & $x$-axis reflection\rule{0pt}{4ex}    
  \end{tabular}
  \caption{Each element in $K_R$ is identified with the composition of the two corresponding maps.}
    \label{tab: kleinbottle}
\end{table}

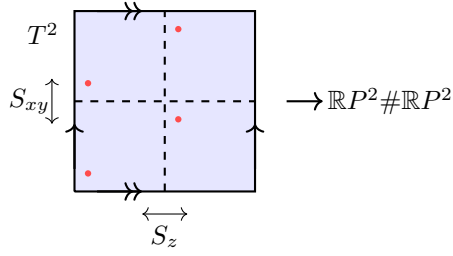
\begin{figure*}[h]
\begin{tikzpicture}[scale=0.6]


\draw [blue!10, fill=blue!10] (-2, -2)--(-2, 2)--(2, 2)--(2,-2)--(-2,-2);

\draw [thick, dashed] (-2,0)--(2,0);
\draw [thick, dashed] (0,-2)--(0,2);
\draw [thick] (-2,-2)--(-2,2)--(2,2)--(2,-2)--(-2,-2);

\draw [thick,->] (-2,-1.5)--(-2,-0.5);

\draw [thick,->] (2,-1.5)--(2,-0.5);

\draw [thick,->>] (-1.5,2)--(-0.5,2);
\draw [thick,->>] (-1.5,-2)--(-0.5,-2);


\draw [thick, ->] (2.7,0)--(3.5,0);
\node at (5,0){$\R P^2\# \R P^2$};
\node at (-2.7, 1.5){$T^2$};

\fill[thick, red!70] (-1.7, -1.6)  circle[radius=2pt];
\fill[thick, red!70] (0.3, -0.4)  circle[radius=2pt];
\fill[thick, red!70] (-1.7, 0.4)  circle[radius=2pt];
\fill[thick, red!70] (0.3, 1.6)  circle[radius=2pt];

\draw [<->] (-2.5, -0.5)--(-2.5,0.5);	
\node at (-3, 0){$S_{xy}$};
\draw [<->] (-0.5, -2.5)--(0.5,-2.5);	
\node at (0, -3){$S_z$};
\end{tikzpicture}
\caption{The torus $T^2$ as a 4-fold cover of the Klein bottle $\R P^2\# \R P^2$}
\label{fig: 4foldcover}	
\end{figure*}

\end{example}

\begin{example}
Consider the monotone toric symplectic manifold $\C P^2\# 2\overline{\C P^2}$ with moment polytope $\Delta$ given on the left in Figure \ref{fig: 2blowup}. Since the real locus of $\C P^2\# 2\overline{\C P^2}$ is diffeomorphic to $\#_3\R P^2$, the standard antisymplectic involution $R^0$ yileds the real Lagrangian $\Fix(R^0)\cong \#_3\R P^2$. We claim that the real Lagrangian $\Fix(R^\sigma)$ associated to the involution $\sigma(x,y)=(y,x)$ is diffeomorphic to $\R P^2$. To see this, recall that by Example \ref{example: cp2} the real Lagrangian $\Fix(R^{\sigma_2})$ in $\C P^2$ is diffeomorphic to $\R P^2$. Since the blow-ups of $\C P^2$ were performed away from the real Lagrangian $\Fix(R^{\sigma_2})$, we deduce that $\Fix(R^\sigma)\cong \R P^2$.
	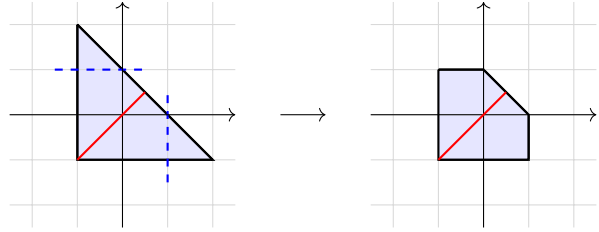
\begin{figure*}[h]
\begin{center}
\begin{tikzpicture}[scale=0.6]

\begin{scope}[xshift=-8cm]

\draw [thick,fill=blue!10] (-1,2)--(0,1)--(1,0)--(2,-1)--(-1,-1)--(-1,2);
\draw[step=1.0,black!15,thin] (-2.5,-2.5) grid (2.5, 2.5);
\draw [thick] (-1,2)--(0,1)--(1,0)--(2,-1)--(-1,-1)--(-1,2);

\draw [thick, dashed,blue] (-1.5,1)--(0.5,1);
\draw [thick, dashed,blue] (1,-1.5)--(1,0.5);

\draw [->] (-2.5,0)--(2.5,0);
\draw [->] (0,-2.5)--(0,2.5);
\draw [thick,red] (-1, -1)--(0.5, 0.5);	
\end{scope}

\draw [thick,fill=blue!10] (-1,1)--(0,1)--(1,0)--(1,-1)--(-1,-1);
\draw[step=1.0,black!15,thin] (-2.5,-2.5) grid (2.5, 2.5);
\draw [thick] (-1,1)--(0,1)--(1,0)--(1,-1)--(-1,-1)--(-1,1);

\draw [->] (-2.5,0)--(2.5,0);
\draw [->] (0,-2.5)--(0,2.5);
\draw [thick,red] (-1, -1)--(0.5, 0.5);

\draw [->] (-4.5,0)--(-3.5,0);
\end{tikzpicture}
\end{center}
\caption{Monotone two-fold blow-up of $\C P^2$}
\label{fig: 2blowup}
\end{figure*}
\end{example}

\begin{example}
Consider the three-fold monotone blow-up of $\C P^2$ as in Figure \ref{fig: threefold}. Then we have
\begin{eqnarray*}
&&	v_1=(1,0),\ v_2=(0,1),\ 
 v_3=(-1,0),\ v_4=(0,-1),\ v_5=(1,1),\ v_6=(-1,-1), \\
&& \kappa_1=\dots=\kappa_6=1.
\end{eqnarray*}
We exhibit the real Lagrangians corresponding to the four involutions,
\begin{eqnarray*}
&&\sigma_0=\begin{pmatrix}
	1 & 0 \\
	0 & 1
\end{pmatrix},\quad \sigma_1=\begin{pmatrix}
	-1 & 0 \\
	0 & -1
\end{pmatrix},\quad \sigma_2=\begin{pmatrix}
	0 & -1 \\
	-1 & 0
\end{pmatrix},\quad \sigma_3=\begin{pmatrix}
	0 & 1 \\
	1 & 0
\end{pmatrix}.
\end{eqnarray*}

\begin{figure*}[h]
\begin{center}
\begin{subfigure}{0.24\textwidth}
  \centering
\begin{tikzpicture}[scale=0.5]
\draw[step=2.0,black!15,thin] (-3,-3) grid (3,3);

\begin{scope}[xshift=-2cm, yshift=-2cm]
\draw [thick, fill=red!70] (0,2)--(2,0)--(4,0)--(4,2)--(2,4)--(0,4)--(0,2);

\draw [thick] (0,2)--(2,0)--(4,0)--(4,2)--(2,4)--(0,4)--(0,2);

\end{scope}

\draw [->] (-3,0)--(3,0);
\draw [->] (0,-3)--(0,3);
\end{tikzpicture}
\caption{$\#_4 \R P^2$}	
\end{subfigure}
\begin{subfigure}{0.24\textwidth}
  \centering
\begin{tikzpicture}[scale=0.5]
\draw[step=2.0,black!15,thin] (-3,-3) grid (3,3);

\begin{scope}[xshift=-2cm, yshift=-2cm]
\draw [thick, fill=blue!10] (0,2)--(2,0)--(4,0)--(4,2)--(2,4)--(0,4)--(0,2);

\draw [thick] (0,2)--(2,0)--(4,0)--(4,2)--(2,4)--(0,4)--(0,2);

\end{scope}

\draw [->] (-3,0)--(3,0);
\draw [->] (0,-3)--(0,3);
\fill[thick, red!70] (0, 0)  circle[radius=4pt];
\end{tikzpicture}
\caption{$T^2$}	
\end{subfigure}
\begin{subfigure}{0.24\textwidth}
  \centering
 \begin{tikzpicture}[scale=0.5]
\draw[step=2.0,black!15,thin] (-3,-3) grid (3,3);

\begin{scope}[xshift=-2cm, yshift=-2cm]
\draw [thick, fill=blue!10] (0,2)--(2,0)--(4,0)--(4,2)--(2,4)--(0,4)--(0,2);

\draw [thick] (0,2)--(2,0)--(4,0)--(4,2)--(2,4)--(0,4)--(0,2);

\end{scope}

\draw [->] (-3,0)--(3,0);
\draw [->] (0,-3)--(0,3);
\draw [red!70,thick] (2,-2) -- (-2,2);
\end{tikzpicture}
\caption{$S^2$} 
\end{subfigure}
\begin{subfigure}{0.24\textwidth}
  \centering
 \begin{tikzpicture}[scale=0.5]
\draw[step=2.0,black!15,thin] (-3,-3) grid (3,3);

\begin{scope}[xshift=-2cm, yshift=-2cm]
\draw [thick, fill=blue!10] (0,2)--(2,0)--(4,0)--(4,2)--(2,4)--(0,4)--(0,2);

\draw [thick] (0,2)--(2,0)--(4,0)--(4,2)--(2,4)--(0,4)--(0,2);

\end{scope}

\draw [->] (-3,0)--(3,0);
\draw [->] (0,-3)--(0,3);
\draw [red!70,thick] (1,1) -- (-1,-1);
\end{tikzpicture}
\caption{$\R P^2\# \R P^2$} 
\end{subfigure}

\end{center}	
\caption{Three-fold monotone blow-up of $\C P^2$}
\label{fig: threefold}
\end{figure*}
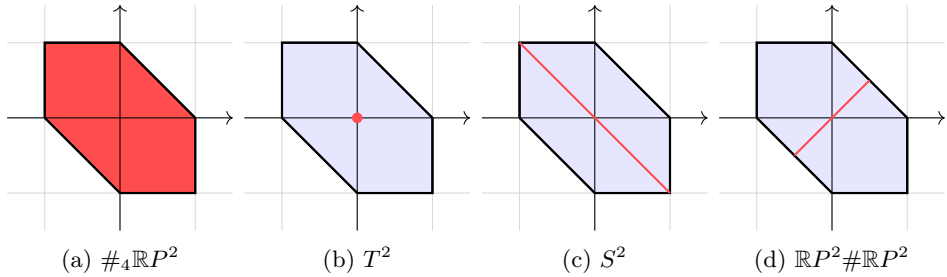
\noindent
{\it Involution $\sigma_1$}. Since $\Fix(\sigma_1)$ is a singleton, we have $\Fix(R^{\sigma_1})=\mu^{-1}\big((0,0)\big)\cong T^2$.\\\\
{\it Involution $\sigma_2$}. Note that the moment polytope of $\C P^2\#3\overline{\C P^2}$ can be seen as the polytope obtained by two fold blow-up of $S^2\times S^2$. Since the real Lagrangian in $S^2\times S^2$ corresponding to the antidiagonal line in the polytope $\Box$ is diffeomorphic to $S^2$ and the blow-ups are performed away from it, we obtain that $\Fix(R^{\sigma_2})\cong S^2$. \\\\
{\it Involution $\sigma_3$}. In a similar vein, since the real Lagrangian in the one point blow up $X_1$ of $\C P^2$, corresponding to the diagonal line, is diffeomorphic to $\R P^2\# \R P^2$, so is $\Fix(R^{\sigma_3})\cong \R P^2\# \R P^2$.
\end{example}

\subsection*{Acknowledgement}
The authors cordially thank Felix Schlenk for careful reading of the first draft. The paper was carried out when the authors visited the Institut de Math\'{e}matiques at Neuch\^{a}tel and the Korea Institute for Advanced Study at Seoul. We are grateful for their warm hospitality. JM specially thanks her advisor Suyoung Choi for continued support and encouragement. JB is supported by the grant 200021-181980/1 of the Swiss National Foundation. JK and JM are supported by Samsung Science and Technology Foundation under Project Number SSTF-BA1901-01. JM is supported by NRF-2019R1A2C2010989.

\bibliographystyle{abbrv}
\bibliography{mybibfile}

\end{document}